\documentclass[11pt]{article}
\usepackage[T1]{fontenc}
\usepackage[english]{babel}
\usepackage{csquotes}
\usepackage{amsmath,amssymb}

\usepackage[margin=1in]{geometry}
\usepackage{amsmath,amssymb,amsthm,mathtools}
\usepackage{mathrsfs}
\usepackage{enumitem}
\usepackage{hyperref}
\usepackage{float}
\usepackage[nameinlink,capitalize]{cleveref}

\usepackage{tikz}
\usetikzlibrary{arrows.meta}

\numberwithin{equation}{section}

\newtheorem{theorem}{Theorem}[section]
\newtheorem{lemma}[theorem]{Lemma}
\newtheorem{proposition}[theorem]{Proposition}
\newtheorem{corollary}[theorem]{Corollary}
\newtheorem{assumption}[theorem]{Assumption}

\theoremstyle{remark}
\newtheorem{remark}[theorem]{Remark}

\newcommand{\R}{\mathbb R}

\title{Analysis of a Conforming Finite Element Method for Second-Harmonic Generation Scattering}
\author{
Ansh Desai\footnote{Department of Mathematics, University of Chicago, Chicago, IL 60637. email: \href{mailto:ansh@uchicago.edu}{ansh@uchicago.edu}.}
\and Peter Monk\footnote{Department of Mathematical Sciences, University of Delaware, Newark, DE 19716. email: \href{mailto:monk@udel.edu}{monk@udel.edu}.}}
\date{}

\begin{document}
\maketitle

\begin{abstract}
In the frequency domain, nonlinear acoustic and electromagnetic wave propagation can,
in certain regimes, be modeled by second-harmonic generation systems consisting of two
coupled nonlinear Helmholtz equations. We analyze a conforming finite element
approximation of a scalar second-harmonic generation scattering problem truncated using an
exact Dirichlet-to-Neumann boundary condition. The discretization uses continuous
piecewise polynomial finite elements of degree \(p\). Under small-data
conditions involving the incident-field and nonlinear susceptibilities, we prove existence and uniqueness of the continuous and discrete nonlinear
solutions in a prescribed small ball. In the same regime, we derive quasi-optimal a priori \(H^1\)-error estimates by a
nonlinear C\'ea-type argument combining linear Galerkin quasi-optimality with
small-data stability estimates for the coupled nonlinearities. We also analyze the convergence of the fixed-point iteration used
to compute the discrete nonlinear solution and quantify the combined effects of finite
element approximation and nonlinear iteration error. A manufactured-solution experiment illustrates the predicted finite
element convergence rates, while additional PML-based scattering computations demonstrate the behavior of the nonlinear
fixed-point solver in two and three dimensions.
\end{abstract}

\section{Introduction}
\label{sec:introduction}
Second-harmonic generation (SHG) is a nonlinear wave phenomenon in which a monochromatic field at a fundamental angular frequency $\omega$ interacts with a nonlinear medium and generates a second-harmonic field at angular frequency $2\omega$.  In the frequency domain, the fields at the two frequencies are coupled: the second harmonic is driven by a quadratic source involving two copies of the fundamental field, while the fundamental field experiences a back-reaction involving the second harmonic and the complex conjugate of the fundamental field. For mathematical discussions of SHG models, see, for example, \cite{DobsonBao,Cakoni3D,AcostaUhlmann,jiang2026}. For an engineering view of nonlinear optics and applications, see~\cite{powers}.

The purpose of this paper is to prove error estimates for a conforming finite element approximation of a model scalar SHG scattering problem. We consider a bounded nonlinear material occupying a domain $D\subset\mathbb R^d$, with $d\in \{2,3\}$, illuminated by a known incident plane wave $u^i$ at the fundamental frequency. The interaction with the nonlinear medium produces a scattered fundamental field $u_1^s$ and a generated second-harmonic field $u_2$. To obtain a bounded-domain formulation suitable for finite element analysis, we truncate the exterior domain with an exact Dirichlet-to-Neumann (DtN) map. This provides a transparent boundary condition on the artificial boundary and removes any truncation error from the analysis, allowing us to focus on the nonlinear SHG coupling and its finite element approximation.

The SHG model arises naturally in electromagnetism. In
\cite{DobsonBao}, Bao and Dobson analyze a one-dimensional SHG model for
a nonlinear optical film and prove existence of a solution by a
fixed-point argument, but do not analyze a finite element discretization.
Cakoni et al.~\cite{Cakoni3D} discuss the derivation of an SHG scattering
model from Maxwell's equations, leading to a coupled system of nonlinear
Helmholtz equations of the form considered here. Their analysis concerns
conditions under which the second-harmonic response does not propagate
outside the nonlinear material, whereas our interest is in the numerical
approximation of the resulting scattering problem. Quadratic
nonlinearities also arise in acoustics. For example, Acosta and
Uhlmann~\cite{AcostaUhlmann} consider a time-domain Westervelt-type
equation for acoustic propagation in biological tissue; a
frequency-domain approximation of such a model leads to a coupled
second harmonic system related to the one studied here.

Earlier numerical studies treated the fully coupled SHG system in one
spatial dimension. Yuan~\cite{Yuan2009} combined a standard finite
element discretization with a fixed-point iteration to simulate
pump-depleted SHG in one-dimensional nonlinear photonic crystals. Bao,
Xu, and Yuan~\cite{BaoXuYuan2011} subsequently introduced an incremental
continuation strategy to improve the robustness of the iteration for
stronger incident fields.

Related finite element analyses have been developed for scalar nonlinear
Helmholtz equations involving Kerr-type pointwise nonlinearities (for example \cite{WuZou2018,Verfuerth2024,angermann2023,angermann2023a}). These studies are not directly applicable to the SHG system, which is nonlocal after eliminating the second-harmonic field.

The study most closely related to the present paper is that of Jiang and
Li~\cite{jiang2026}, who analyze nonlinear Helmholtz systems arising in
second- and third-harmonic generation. They truncate the exterior domain
by a perfectly matched layer and discretize the resulting problem using
a linear continuous interior penalty finite element method. Their analysis includes preasymptotic finite element error
estimates and convergence results for several nonlinear iterative
schemes. In contrast, our theoretical formulation uses 
a standard conforming Galerkin discretization and differs in the techniques
used to obtain error estimates.

The novel contribution of this paper is an error analysis of a conforming
finite element method for the coupled SHG scattering problem. Under a
product-type small-data condition involving the nonlinear
susceptibilities, the stability constants of the two linear DtN
problems, and the incident-field data, we prove existence and uniqueness
of both the continuous and discrete nonlinear problems by a unified
contraction argument. We then establish a quasi-optimal a priori error estimate in the
\(H^1\)-norm using a nonlinear C\'ea-type argument that 
works directly with the finite element solution. This is simpler than the
approach in~\cite{jiang2026} which involves analysis of an iterative solution. We also derive convergence estimates for the
discrete fixed-point iterates that separate finite element approximation
error from nonlinear iteration error. We restrict attention to a
nonlinearity supported in a volume domain; surface SHG models such as
those considered in~\cite{AndreyBogdanov} lie beyond the scope of this
work.

Our theoretical analysis uses the exact DtN map and therefore contains no
error from approximating that map. If the DtN operator is truncated in a
numerical implementation, the finite element estimates must be augmented
by the corresponding DtN approximation error; related estimates are given
in \cite{angermann2023,angermann2023a}. In the computations of
Section~\ref{sec:numerics}, we instead truncate the exterior domain with a
PML. An analysis of PML truncation for SHG systems is given by Jiang and
Li~\cite{jiang2026}.

The paper is organized as follows. In Section~\ref{sec:continuous}, we introduce the notation, formulate the coupled SHG equations, truncate the exterior scattering problem using the DtN map, and establish the continuous linear invertibility needed in the nonlinear analysis. In Section~\ref{sec:fem}, we state the assumptions on the conforming finite element spaces and prove existence and uniqueness of the continuous and discrete problems under small-data by a single contraction argument following the strategy of \cite{DobsonBao}. We then establish quasi-optimal finite element error estimates and corresponding estimates for the iterates of a fixed-point scheme. In Section~\ref{sec:numerics}, we present a numerical experiment verifying the predicted convergence rate, as well as scattering examples in two and three dimensions. The paper concludes in Section~\ref{sec:conclusion}.

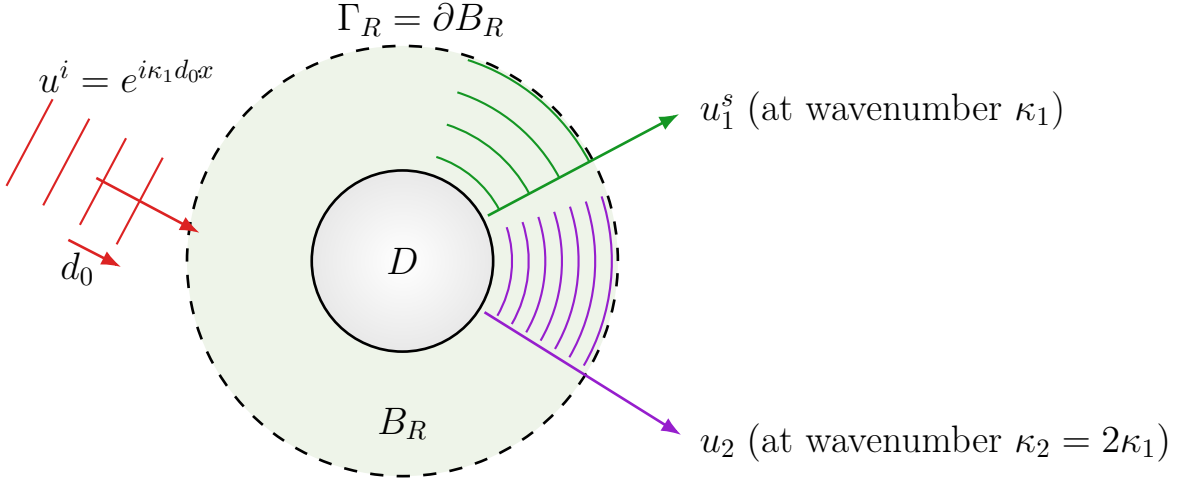
\begin{figure}[t]
\centering

\usetikzlibrary{arrows.meta,calc}

\begin{tikzpicture}[
    >=Latex,
    every node/.style={font=\Large},
    fieldarrow/.style={->, line width=1.0pt},
    wavefront/.style={line width=0.8pt}
]

\definecolor{outerfill}{RGB}{238,244,233}
\definecolor{incidentred}{RGB}{220,35,35}
\definecolor{scatteredgreen}{RGB}{20,150,35}
\definecolor{secondarypurple}{RGB}{150,30,205}

\def\R{2.85}
\def\rD{1.20}
\fill[outerfill] (0,0) circle (\R);
\draw[dashed, line width=1.0pt, dash pattern=on 5pt off 5pt] (0,0) circle (\R);

\shade[inner color=white, outer color=gray!18] (0,0) circle (\rD);
\draw[line width=1.0pt] (0,0) circle (\rD);
\node at (0,0) {$D$};

\node at (0,-2.15) {$B_R$};
\node[anchor=west] at (-1,3.2) {$\Gamma_R=\partial B_R$};

\node[anchor=west] at (-4.95,2.45) {$u^i=e^{i\kappa_1 d_0\!\cdot\! x}$};

\begin{scope}[shift={(-4.05,1.10)}, rotate=-28]
    \draw[fieldarrow, incidentred] (0,0) -- (1.55,0);
    \foreach \x in {-1.00,-0.45,0.10,0.65}
        \draw[wavefront, incidentred] (\x,-0.65) -- (\x,0.65);
\end{scope}

\node at (-4.30,-0.10) {$d_0$};
\draw[fieldarrow, incidentred] (-4.42,0.28) -- (-3.72,-0.08);

\foreach \rad in {1.45,1.90,2.35,2.80} {
    \draw[wavefront, scatteredgreen]
        (28:\rad) arc[start angle=28,end angle=72,radius=\rad];
}

\draw[fieldarrow, scatteredgreen] (28:1.28) -- (28:4.15);
\node[anchor=west] at (3.8,2) {$u_1^s\; (\mbox{at wavenumber }\kappa_1$)};

\foreach \rad in {1.45,1.67,1.89,2.11,2.33,2.55,2.77} {
    \draw[wavefront, secondarypurple]
        (-30:\rad) arc[start angle=-30,end angle=18,radius=\rad];
}
\draw[fieldarrow, secondarypurple] (-32:1.28) -- (-32:4.35);
\node[anchor=west] at (3.8,-2.4) {$u_2\;(\mbox{at wavenumber }\kappa_2=2\kappa_1)$};

\end{tikzpicture}

\caption{Schematic of the SHG scattering problem. An incident wave $(u^i)$
at the fundamental wavenumber $\kappa_1$ impinges on the region $D$ containing the nonlinearity.
The interaction of this wave with the medium generates a scattered field ($u_1^s$) at the fundamental frequency and
a second-harmonic field ($u_2$) at wavenumber $\kappa_2=2\kappa_1$. The exact DtN
formulation truncates the exterior domain at the artificial boundary
$\Gamma_R=\partial B_R$.}
\label{fig:shg-scattering}
\end{figure}
\section{The continuous SHG problem}
\label{sec:continuous}

Since the one-dimensional model is treated in \cite{DobsonBao}, we assume that $d\in\{2,3\}$. Let \(D\subset \mathbb R^d\) be a bounded Lipschitz domain occupied by the nonlinear
material. We assume that the material coefficients at the fundamental and second-harmonic
frequencies are described by real-valued squared refractive indices \(n_1\) and
\(n_2\). To simplify wave propagation outside the nonlinear medium, we assume that the
background is homogeneous so that $
    n_1=n_2=1
    $ in $\mathbb R^d\setminus \overline D.
$
Because $D$ is bounded, we can choose \(R>0\) sufficiently large such that
$
    \overline D\subset B_R$,
where \(B_R\) is the ball of radius \(R\) centered at the origin (see Fig.~\ref{fig:shg-scattering}). The boundary of $B_R$ is denoted $\Gamma_R:=\partial B_R$. We use the same notation
\(n_j\), $j=1,2$, for the restrictions of the full-space coefficients to \(B_R\), and assume that there are constants $n_{\rm min}$ and $n_{\rm max}$ with
\[
    n_j\in L^\infty(B_R;\mathbb R),
    \qquad
    0< n_{\rm min}\le n_j(x)\le n_{\rm max}<\infty
    \quad\text{for a.e. }x\in B_R,
    \qquad j=1,2.
\]

Let \(c_0\) denote the background wave speed. If \(\omega>0\) is the angular frequency of
the fundamental mode, then the corresponding background wavenumbers are
\[
    \kappa_1:=\frac{\omega}{c_0},
    \qquad
    \kappa_2:=2\kappa_1=\frac{2\omega}{c_0}.
\]
We prescribe an incident plane wave at the fundamental wavenumber $\kappa_1$ propagating in the direction $d_0\in \mathbb{S}^{d-1}$ given by
\[
    u^i(x)=\exp(i\kappa_1 d_0\cdot x)
\]
This impinges on the scatterer $D$ containing the nonlinear medium. Let \(u_1\) denote the total field at
the fundamental frequency, \(u_1^s\) the corresponding scattered field, and \(u_2\) the
outgoing second-harmonic field, so that 
\begin{equation}
    u_1=u^i+u_1^s
    \qquad\text{in }\mathbb R^d.\label{u1tot}
\end{equation}
The full-space SHG scattering problem is to find $(u_1,u_2)\in H^1_{\rm loc}(\mathbb{R}^d)\times H^1_{\rm loc}(\mathbb{R}^d)$ such that
\begin{align}
    \Delta u_1+\kappa_1^2 n_1 u_1
    &=
    -\chi_1 \overline{u_1}u_2
    &\text{in }\mathbb R^d,
    \label{eq:full-shg-u1}\\
    \Delta u_2+\kappa_2^2 n_2 u_2
    &=
    -\chi_2 u_1^2
    &\text{in }\mathbb R^d,
    \label{eq:full-shg-u2}
\end{align}
together with the Sommerfeld radiation conditions
\begin{align}
    \lim_{r=|x|\to \infty}r^{(d-1)/2}
    \left(
        \frac{\partial u_1^s}{\partial r}
        -i\kappa_1 u_1^s
    \right)
    &=0,
    \label{eq:rad-u1}\\
    \lim_{r=|x|\to \infty}r^{(d-1)/2}
    \left(
        \frac{\partial u_2}{\partial r}
        -i\kappa_2 u_2
    \right)
    &=0,
    \label{eq:rad-u2}
\end{align}
holding uniformly in \(\widehat x=x/|x|\in \mathbb S^{d-1}\). Here, \(\overline{u_1}\) denotes the complex conjugate of \(u_1\). The susceptibility
coefficients \(
    \chi_j\in L^\infty(\mathbb R^d)\), \(j=1,2,
\)
are real-valued, nonnegative, scaled second-order nonlinear susceptibilities supported in $\overline{D}$.
Thus the nonlinear interaction is confined to the material region \(D\).
Equations \eqref{u1tot}-\eqref{eq:rad-u2} govern the exact
nonlinear scattering problem that we shall approximate (we refer to it as the ``continuous'' problem to distinguish it from the ``discrete'' finite element problem). This model problem is posed on all of $\mathbb{R}^d$ and needs to be truncated to a bounded domain for finite element analysis.  For simplicity, we do this next using the Dirichlet-to-Neumann (DtN) map.

To define the DtN map, let \(k>0\) and \(g\in H^{1/2}(\Gamma_R)\). Let
\(w_g\) denote the outgoing solution of the exterior Helmholtz problem
\begin{align}
    \Delta w_g+k^2w_g&=0
    &&\text{in }\mathbb R^d\setminus \overline{B_R},\label{wg1}\\
    w_g&=g
    &&\text{on }\Gamma_R,\\
    \lim_{r=|x|\to\infty}
    r^{(d-1)/2}
    \left(
        \frac{\partial w_g}{\partial r}-ikw_g
    \right)&=0.\label{wg-bc}
\end{align}
This exterior problem is well posed for each \(k>0\) \cite[Theorem 2.6.6]{Nedelec}. The outgoing Dirichlet-to-Neumann map is defined by
\[
    T_k:H^{1/2}(\Gamma_R)\to H^{-1/2}(\Gamma_R),
    \qquad
    T_kg:=\partial_\nu w_g|_{\Gamma_R},
\]
where \(\nu\) denotes the unit outward normal to \(B_R\). For each $k>0$, the DtN map is well-defined and bounded. For a more detailed discussion of the DtN map, we refer to \cite{grasslee2025dirichlettoneumannoperatorhelmholtzproblem}.
Here and below, we set
\[
    V:=H^1(B_R;\mathbb C),
    \qquad
    \|u\|_V:=\|u\|_{H^1(B_R)}.
\]
The trace operator is denoted by
\(
    \gamma:V\to H^{1/2}(\Gamma_R).
\)
Throughout, \(V'\) denotes the anti-dual of \(V\), i.e., the space of
continuous conjugate-linear functionals on \(V\). The duality pairing
\(\langle F,v\rangle\) is linear in \(F\) and conjugate-linear in \(v\). We also denote the conjugate-linear duality pairing between
\(H^{-1/2}(\Gamma_R)\) and \(H^{1/2}(\Gamma_R)\) by
\(\langle\cdot,\cdot\rangle_{\Gamma_R}\).

We now restrict the full-space fields in \eqref{eq:full-shg-u1}-\eqref{eq:rad-u2} to \(B_R\). Since \(n_j=1\) and
\(\chi_j=0\) outside \(D\), the exterior fields solve homogeneous Helmholtz
equations in \(\mathbb R^d\setminus \overline{B_R}\). The radiation conditions
can therefore be encoded exactly by DtN maps, and we can eliminate the scattered field. The full-space
scattering problem \eqref{eq:full-shg-u1}-\eqref{eq:rad-u2} is equivalent, in \(B_R\), to the following truncated
problem on $B_R$: find the total fields $(u_1,u_2)\in V\times V$ such that:
\begin{align}
    \Delta u_1+\kappa_1^2n_1u_1
    &=
    -\chi_1\overline{u_1}u_2
    &&\text{in }B_R,
    \label{eq:u1strong}\\
    \Delta u_2+\kappa_2^2n_2u_2
    &=
    -\chi_2u_1^2
    &&\text{in }B_R,
    \label{eq:u2strong}\\
    \partial_\nu u_1
    &=
    T_{\kappa_1}\gamma(u_1-u^i)+\partial_\nu u^i
    &&\text{on }\Gamma_R,
    \label{eq:u1-dtn-bc}\\
    \partial_\nu u_2
    &=
    T_{\kappa_2}\gamma u_2
    &&\text{on }\Gamma_R.
    \label{eq:u2-dtn-bc}
\end{align}
Conversely, any solution of
\eqref{eq:u1strong}-\eqref{eq:u2-dtn-bc} can be extended to a solution of
\eqref{eq:full-shg-u1}-\eqref{eq:rad-u2}. Specifically, solve
\eqref{wg1}-\eqref{wg-bc} with
\(g=\gamma(u_1-u^i)\) and \(k=\kappa_1\) for the scattered fundamental
field, and with \(g=\gamma u_2\) and \(k=\kappa_2\) for the
second-harmonic field. The boundary conditions
\eqref{eq:u1-dtn-bc}-\eqref{eq:u2-dtn-bc} ensure continuity of the
normal derivatives across \(\Gamma_R\).

Since \(T_{\kappa_1}\) is linear, the first boundary condition \eqref{eq:u1-dtn-bc}
may be written as
\[
    \partial_\nu u_1
    =
    T_{\kappa_1}\gamma u_1+f_R
    \qquad\text{on }\Gamma_R,
\]
where
\[
    f_R:=\partial_\nu u^i-T_{\kappa_1}(\gamma u^i)
    \in H^{-1/2}(\Gamma_R).
\]
We define the corresponding boundary functional
\[
    \ell(\xi)
    :=
    \langle f_R,\gamma\xi\rangle_{\Gamma_R},
    \qquad \xi\in V.
\]
Since \(u^i\) is smooth and \(T_{\kappa_1}\) is bounded, we have
$
    \ell\in V'.
$

Multiplying \eqref{eq:u1strong} and \eqref{eq:u2strong} by the complex conjugate of smooth test functions,
integrating by parts over \(B_R\), and using the DtN boundary conditions
\eqref{eq:u1-dtn-bc}-\eqref{eq:u2-dtn-bc}, we obtain the weak formulation of the SHG scattering problem:
find
\(
    (u_1,u_2)\in V\times V
\)
such that
\begin{align}
    a_1(u_1,\xi)
    &=
    \int_{B_R}
        \chi_1\overline{u_1}u_2\overline{\xi}\,dx
    +
    \ell(\xi)
    &&\forall \xi\in V,
    \label{eq:continuous-shg-1}\\
    a_2(u_2,\eta)
    &=
    \int_{B_R}
        \chi_2u_1^2\overline{\eta}\,dx
    &&\forall \eta\in V.
    \label{eq:continuous-shg-2}
\end{align}
Here, for \(j=1,2\), we define the sesquilinear forms
\(a_j:V\times V\to\mathbb C\) by
\begin{align}
    a_1(w,\xi)
    &:=
    \int_{B_R}
    \left(
        \nabla w\cdot\nabla\overline{\xi}
        -
        \kappa_1^2n_1w\overline{\xi}
    \right)\,dx
    -
    \langle T_{\kappa_1}\gamma w,\gamma\xi\rangle_{\Gamma_R},
    \label{eq:a1-cont}\\
    a_2(w,\eta)
    &:=
    \int_{B_R}
    \left(
        \nabla w\cdot\nabla\overline{\eta}
        -
        \kappa_2^2n_2w\overline{\eta}
    \right)\,dx
    -
    \langle T_{\kappa_2}\gamma w,\gamma\eta\rangle_{\Gamma_R}.
    \label{eq:a2-cont}
\end{align}

The continuous and discrete nonlinear problems will be analyzed simultaneously
in Theorem~\ref{thm:continuous-discrete-small-data}. In preparation for that
argument, we record three properties of the continuous problem: boundary
unique continuation for the linear Helmholtz operators, estimates for the quadratic coupling terms, and invertibility of the linear Helmholtz operators.

The uniqueness argument for the linear problem uses the following standard boundary unique-continuation property.
\begin{lemma}[Boundary unique continuation]
\label{lem:boundary-ucp}
Let \(d\in\{2,3\}\), let \(n_j\in L^\infty(B_R)\), and fix
\(j\in\{1,2\}\). Suppose that \(u\in V\) satisfies
\[
        \Delta u+\kappa_j^2 n_j u=0
        \quad\text{in }B_R
\]
in the sense of distributions, and that
$
        \gamma u=0,
        $ and $
        \partial_\nu u=0
        \quad\text{on }\Gamma_R
$
in the trace sense. Then \(u=0\) in \(B_R\).
\end{lemma}

In addition, we recall that by using the Sobolev embedding theorem~(see for example \cite[Corollary 9.14]{brezis}), we have the embedding $H^1(B_R)\hookrightarrow L^4(B_R)$ when $d\in \{2,3\}$, and we obtain that
$$
    \left|
    \int_{B_R}wz\overline{\phi}\,dx
    \right|
    \le
    \|w\|_{L^4(B_R)}
    \|z\|_{L^4(B_R)}
    \|\phi\|_{L^2(B_R)}\le C_{\rm prod}\|w\|_V\|z\|_V\|\phi\|_V.
$$
Taking the supremum over all $\phi\in V$ with $\|\phi\|_V=1$ gives the following product estimate:
\begin{lemma}[Product estimates]
\label{lem:product-estimate-continuous}
There exists a constant
$C_{\rm prod}>0$, depending only on $B_R$ and $d$, such that for all
$w,z,\widetilde w,\widetilde z\in V$,
\begin{align}
    \|wz\|_{V'}
    &\le
    C_{\rm prod}\|w\|_V\|z\|_V,
    \label{eq:product-estimate-1}\\
    \|wz-\widetilde w\widetilde z\|_{V'}
    &\le
    C_{\rm prod}
    \left(
        \|w\|_V\|z-\widetilde z\|_V
        +
        \|\widetilde z\|_V\|w-\widetilde w\|_V
    \right).
    \label{eq:product-estimate-2}
\end{align}
The same estimates hold if any factor is replaced by its complex conjugate.
\end{lemma}

Let $A_j:V\to V'$ be the operators induced by $a_j$, namely,
\begin{equation*}
        \langle A_jw,\phi\rangle:=a_j(w,\phi),
    \qquad w,\phi\in V,\; j=1,2.
\end{equation*}
We use the following standard well-posedness result for the linear problem.
\begin{proposition}[Linear invertibility]
\label{prop:linear-continuous-invertibility}
For $j\in \{1,2\}$,
$A_j:V\to V'$ is an isomorphism. Consequently, there exists $M_j>0$ such that
$$
    \|A_j^{-1}g\|_V\le M_j\|g\|_{V'}
    \qquad\forall g\in V'.
$$
\end{proposition}
Proposition~\ref{prop:linear-continuous-invertibility} is the standard
Fredholm well-posedness result for the linear Helmholtz problem with the
outgoing DtN boundary condition; see, for example,
\cite[Theorem~3.8]{MelenkSauter} and
\cite[Theorem~10]{angermann2023}. The usual proof combines a
G{\aa}rding inequality and compactness to obtain a Fredholm operator of
index zero. Injectivity follows from the outgoing DtN identity and
Lemma~\ref{lem:boundary-ucp}. We omit the details.
\begin{remark} \label{k-dep} The stability constants $M_j$ may depend on $\kappa_j$, $n_j$, and the scattering geometry.  In \cite{EsterhazyMelenk}, it is proved that, for a Lipschitz scatterer, there exist constants $\alpha,C>0$ independent of $\kappa_j>0$ such that $M_j\le C\kappa_j^\alpha$ as $\kappa_j\to\infty$. The authors further discuss geometric conditions under which this bound holds (the dependence can grow much more rapidly for more exotic scatterers such as trapping geometries).  Because our goal is to analyze the nonlinear problem, we do not trace the $\kappa_j$-dependence of constants. 
\end{remark}
With this notation, \eqref{eq:continuous-shg-1}-\eqref{eq:continuous-shg-2} are equivalent to
$$
    A_1u_1
    =
    \ell+\chi_1\overline{u_1}u_2,
    \qquad
    A_2u_2
    =
    \chi_2u_1^2,
$$
where the nonlinear terms are interpreted as elements of \(V'\) according
to Lemma~\ref{lem:product-estimate-continuous}. Under Proposition~\ref{prop:linear-continuous-invertibility}, the second equation may be eliminated, letting us write a scalar equation for \(u_1\):
\[
A_1u_1
=
\ell+\chi_1\overline{u_1}\,A_2^{-1}(\chi_2u_1^2).
\]
This form is useful for the fixed-point argument. It also shows why the problem is not a direct instance of the scalar framework in~\cite{angermann2023} where the nonlinear terms are generated by pointwise nonlinearities in a scalar field. Here, after eliminating \(u_2\), the nonlinear response contains the Helmholtz resolvent \(A_2^{-1}\). Consequently, the value of the reduced nonlinearity at a point depends on the global solution of the second-harmonic Helmholtz problem driven by \(u_1^2\). Although the coupled SHG system is local in \((u_1,u_2)\), the reduced scalar equation for \(u_1\) contains a nonlocal nonlinear operator.

Proposition~\ref{prop:linear-continuous-invertibility} supplies the
continuous linear stability estimate needed in the fixed-point argument.
Its discrete counterpart is recorded in
Proposition~\ref{prop:linear-galerkin-stability}, after which the two
nonlinear problems are treated together in
Theorem~\ref{thm:continuous-discrete-small-data}.

\section{Analysis of the finite element method}
\label{sec:fem}
We now analyze a conforming Galerkin approximation of the exact DtN truncated problem
\eqref{eq:continuous-shg-1}-\eqref{eq:continuous-shg-2}. Throughout this
section we retain the notation of Section~\ref{sec:continuous}. The theory below is for the exact DtN map on the auxiliary boundary. This assumption allows us to focus on the
nonlinear finite element error. If the DtN operators are replaced by finite
spherical-harmonic or trigonometric polynomial  truncations, then the final estimates must be augmented by
DtN truncation errors of the kind studied for linear exterior Helmholtz
problems in, for example,  \cite{Koyama}. If the curved artificial boundary is
approximated by a polygonal or polyhedral boundary, then additional geometric
consistency errors are present. Rather than a polygonal approximation, one may use mapped or curved
finite element spaces on the exact ball, as in the hp-DtN framework of \cite{MelenkSauter}. We do not develop those additional
terms in the error analysis here.

\subsection{The grid and the discrete variational formulation}
\label{subsec:grid-discrete-formulation}

Let $\{\mathcal T_h\}_{h>0}$ be a family of shape-regular conforming meshes of triangles ($d=2$) or tetrahedra ($d=3$) covering $B_R$  (of necessity, elements sharing edges or faces with $\partial B_R$ will be curvilinear), and
let
    $V_h\subset V$ 
be a conforming finite element space of piecewise polynomials of polynomial degree $p\ge 1$ (see \cite{MelenkSauter}). For the remainder of the paper, we assume
that the family $\{V_h\}$ is dense in $V$, i.e.,
$$
    \inf_{z_h\in V_h}\|v-z_h\|_V\to 0
    \qquad\text{as }h\to0
    \qquad\forall v\in V.
$$
For convergence rates, we additionally assume the following  approximation
property.

\begin{assumption}
\label{ass:fe-approx}
There exists a constant $C_{\rm app}>0$, independent of $h$, such that for
every $s\in(0,p]$ and every $v\in H^{1+s}(B_R)$,
$$
    \inf_{z_h\in V_h}\|v-z_h\|_{V}
    \le
    C_{\rm app}h^s |v|_{H^{1+s}(B_R)}.
$$
\end{assumption}

\begin{remark}
Since $B_R$ has a curved boundary, Assumption~\ref{ass:fe-approx} should be
understood as an assumption on the construction of the conforming spaces $V_h\subset V$.  It
may be realized using an interface-fitted mesh~\cite{angermann2023a},  or patch-wise mapping~\cite{MelenkSauter,MelenkSauter1} since the outer boundary is either a circle or  sphere. 
\end{remark}
The discrete nonlinear problem is: find
$
    (u_{1,h},u_{2,h})\in V_h\times V_h
$
such that
\begin{align}
    a_1(u_{1,h},\xi_h)
    &=
    \int_{B_R}
        \chi_1\overline{u_{1,h}}u_{2,h}\overline{\xi_h}\,dx
    +
    \ell(\xi_h)
    &&\forall \xi_h\in V_h,
    \label{eq:discrete-shg-1}\\
    a_2(u_{2,h},\eta_h)
    &=
    \int_{B_R}
        \chi_2u_{1,h}^2\overline{\eta_h}\,dx
    &&\forall \eta_h\in V_h.
    \label{eq:discrete-shg-2}
\end{align}

\subsection{Continuous and discrete well-posedness}
\label{subsec:fem-existence}

We first recall the linear Galerkin stability and quasi-optimality statement
needed in the nonlinear estimates (see \cite[Theorem 4.3]{MelenkSauter}):

\begin{proposition}[Linear stability and quasi-optimality]
\label{prop:linear-galerkin-stability}
There exist constants $h_0>0$, $M_{\rm lin}>0$,
and $C_{\rm qo}>0$ such that for each $j\in\{1,2\}$ and every $0<h\le h_0$
the following hold:
\begin{enumerate}
\item For every $g_h\in V_h'$, the discrete linear problem
\begin{equation}
    a_j(w_h,\phi_h)
    =
    \langle g_h,\phi_h\rangle
    \qquad\forall \phi_h\in V_h\label{ajh}
\end{equation}
has a unique solution $w_h\in V_h$, and
\begin{equation}
    \|w_h\|_V
    \le
    M_{\rm lin}\|g_h\|_{V_h'}.
\label{discrete_stab}
\end{equation}

\item For every $g\in V'$, let $w\in V$ solve
$$
    a_j(w,\phi)=\langle g,\phi\rangle
    \qquad\forall \phi\in V.
$$
Let \(g_h:=g|_{V_h}\in V_h'\), and let \(w_h\in V_h\) be the
solution of \eqref{ajh} with right-hand side \(g_h\). Then we have the following quasi-optimal estimate:
$$
    \|w-w_h\|_V
    \le
    C_{\rm qo}
    \inf_{z_h\in V_h}\|w-z_h\|_V.
$$
\end{enumerate}
\end{proposition}
\begin{remark} 
\label{k-dep-2}
The constant $M_{\rm lin}$ depends on the wavenumber $\kappa_j$ as well as the shape of the scatterer.  The constant $h_0$ also depends on $\kappa_j$.  As $\kappa_j$ increases, it is necessary to decrease $h_0$ to control the number of unknowns per wavelength, and further decreased to control pollution error~\cite{ihlenburg}.  Thus, the grid must be chosen based on the larger wavenumber $\kappa_2$ to resolve the second harmonic.
\end{remark}

Now we define some constants that appear in the upcoming analysis. Set
\begin{equation}
\label{eq:uniform-linear-stability-constants}
    \widehat M_j:=\max\{M_j,M_{\rm lin}\},
    \qquad j=1,2.
\end{equation}
For $r>0$, set
\begin{equation}
\label{eq:small-data-constants}
\begin{aligned}
    \alpha_r
    &:=\widehat M_1C_{\rm prod}\|\chi_1\|_{L^\infty(B_R)}r,
    &
    \beta_r
    &:=2\widehat M_2C_{\rm prod}\|\chi_2\|_{L^\infty(B_R)}r,\\
    \sigma_r
    &:=\frac{\beta_r r}{2},
    &
    \theta_r
    &:=\widehat M_1C_{\rm prod}\|\chi_1\|_{L^\infty(B_R)}\sigma_r
      =\frac{\alpha_r\beta_r}{2},\\
    q_r
    &:=\theta_r+\alpha_r\beta_r
      =\frac{3}{2}\alpha_r\beta_r.
\end{aligned}
\end{equation}
For $j=1,2$, define $A_{j,h}:V_h\to V_h'$ by
\[
    \langle A_{j,h}w_h,\varphi_h\rangle
    :=a_j(w_h,\varphi_h),
    \qquad w_h,\varphi_h\in V_h,
\]
and define $R_h:L^2(B_R)\to V_h'$ and
\(\ell_h\in V_h'\) by
\begin{equation}
\label{eq:discrete-data-operators}
    \langle R_hf,\varphi_h\rangle
    :=\int_{B_R}f\overline{\varphi_h}\,dx,
    \qquad
    \ell_h(\varphi_h):=\ell(\varphi_h).
\end{equation}
Thus \eqref{eq:discrete-shg-1}-\eqref{eq:discrete-shg-2} are equivalent to
\begin{equation}
\label{hequation}
    A_{1,h}u_{1,h}
    =\ell_h+R_h(\chi_1\overline{u_{1,h}}u_{2,h}),
    \qquad
    A_{2,h}u_{2,h}=R_h(\chi_2u_{1,h}^2).
\end{equation}

For $\rho>0$, introduce the rectangular balls
\begin{equation}
\label{eq:continuous-discrete-balls}
\begin{aligned}
    \mathcal B_\rho
    &:=\left\{(u,v)\in V\times V:
        \|u\|_V\le\rho,\ \|v\|_V\le\sigma_\rho\right\},\\
    \mathcal B_{\rho,h}
    &:=\left\{(u_h,v_h)\in V_h\times V_h:
        \|u_h\|_V\le\rho,\ \|v_h\|_V\le\sigma_\rho\right\}.
\end{aligned}
\end{equation}
\begin{theorem}[Continuous and discrete small-data well-posedness]
\label{thm:continuous-discrete-small-data}
Assume that the Galerkin stability estimate \eqref{discrete_stab} holds for
all $0<h\le h_0$. Let $\rho>0$ satisfy
\begin{equation}
\label{eq:small-data-radius}
    \widehat M_1\|\ell\|_{V'}+\theta_\rho\rho\le\rho,
    \qquad
    q_\rho<1.
\end{equation}
Then the continuous nonlinear problem
\eqref{eq:continuous-shg-1}-\eqref{eq:continuous-shg-2} has a unique
solution $(u_1,u_2)\in\mathcal B_\rho$. Moreover, for every
$0<h\le h_0$, the discrete nonlinear problem
\eqref{eq:discrete-shg-1}-\eqref{eq:discrete-shg-2} has a unique
solution $(u_{1,h},u_{2,h})\in\mathcal B_{\rho,h}$.
\end{theorem}
\begin{remark}[Small data condition]
\label{rem:small-data-interpretation}
Let
\[
    a:=\widehat M_1\|\ell\|_{V'},
    \qquad
    \mu:=\widehat M_1\widehat M_2C_{\rm prod}^2
    \|\chi_1\|_{L^\infty(B_R)}
    \|\chi_2\|_{L^\infty(B_R)}.
\]
By \eqref{eq:small-data-constants}, condition
\eqref{eq:small-data-radius} is equivalent to
\begin{equation}
\label{eq:small-data-radius-expanded}
    a+\mu\rho^3\le\rho,
    \qquad
    3\mu\rho^2<1.
\end{equation}
There exists a radius $\rho>0$ satisfying
\eqref{eq:small-data-radius-expanded} if and only if
\begin{equation}
\label{eq:coefficient-smallness}
    \widehat M_1^{\,3}\widehat M_2C_{\rm prod}^2\|\ell\|_{V'}^2
    \|\chi_1\|_{L^\infty(B_R)}
    \|\chi_2\|_{L^\infty(B_R)}
    <\frac{4}{27}.
\end{equation}
Indeed, when $a>0$, \eqref{eq:coefficient-smallness} implies that
\(\rho=3a/2\) satisfies \eqref{eq:small-data-radius-expanded}. Conversely,
the maximum of $r-\mu r^3$ over $r\ge0$ is
\(2/(3\sqrt{3\mu})\) when $\mu>0$, and equality occurs where
$3\mu r^2=1$; thus the strict contraction condition makes
\eqref{eq:coefficient-smallness} necessary. If $a=0$, choose a sufficiently
small positive radius, while if $\mu=0$, choose any positive radius satisfying
$\rho\ge a$. The constants
$\widehat M_1$ and $\widehat M_2$, and hence the threshold,
are frequency dependent as discussed in
Remarks~\ref{k-dep} and \ref{k-dep-2}.
\end{remark}
\begin{proof}
We prove both assertions simultaneously. Let $X=V$ for the continuous
problem and $X=V_h$ for the discrete problem, where $0<h\le h_0$.
In the first case, let $A_{j,X}=A_j$, $J_X=J$, and
$\ell_X=\ell$, where $J:L^2(B_R)\to V'$ is the canonical map
defined by
\[
    \langle Jf,\varphi\rangle
    :=\int_{B_R}f\overline\varphi\,dx.
\]
In the second case, let $A_{j,X}=A_{j,h}$, $J_X=R_h$, and
$\ell_X=\ell_h$. Propositions~\ref{prop:linear-continuous-invertibility}
and~\ref{prop:linear-galerkin-stability}, together with
\eqref{eq:discrete-data-operators}, give
\begin{equation}
\label{eq:common-linear-bounds}
    \|A_{j,X}^{-1}g\|_V\le \widehat M_j\|g\|_{X'},
    \qquad
    \|\ell_X\|_{X'}\le\|\ell\|_{V'}.
\end{equation}
Furthermore, Lemma~\ref{lem:product-estimate-continuous} implies
\begin{equation}
\label{eq:project-bound}
    \|J_X(\chi_jwz)\|_{X'}
    \le C_{\rm prod}\|\chi_j\|_{L^\infty(B_R)}
        \|w\|_V\|z\|_V,
\end{equation}
for $w,z\in X$, with the same estimate when either factor is conjugated.

Set
\[
    \mathcal B_{\rho,X}^1:=\{u\in X:\|u\|_V\le\rho\}
\]
and define
\begin{equation}
\label{eq:common-fixed-point-maps}
\begin{aligned}
    \mathcal G_X(u)
    &:=A_{2,X}^{-1}J_X(\chi_2u^2),\\
    \mathcal F_X(u)
    &:=A_{1,X}^{-1}
       \left(\ell_X+J_X(\chi_1\overline u\,\mathcal G_X(u))\right).
\end{aligned}
\end{equation}
For $u\in\mathcal B_{\rho,X}^1$,
\eqref{eq:common-linear-bounds}, \eqref{eq:project-bound}, and
\eqref{eq:small-data-constants} yield
\begin{equation}
\label{eq:common-map-bounds}
\begin{aligned}
    \|\mathcal G_X(u)\|_V
    &\le \widehat M_2C_{\rm prod}\|\chi_2\|_{L^\infty(B_R)}\rho^2
      =\sigma_\rho,\\
    \|\mathcal F_X(u)\|_V
    &\le \widehat M_1\|\ell\|_{V'}
      +\widehat M_1C_{\rm prod}\|\chi_1\|_{L^\infty(B_R)}
       \rho\sigma_\rho\\
    &=\widehat M_1\|\ell\|_{V'}+\theta_\rho\rho
      \le\rho.
\end{aligned}
\end{equation}
Hence \(\mathcal F_X\) maps \(\mathcal B_{\rho,X}^1\) into itself.

Let $u,\widetilde u\in\mathcal B_{\rho,X}^1$, and set
$v=\mathcal G_X(u)$ and
$\widetilde v=\mathcal G_X(\widetilde u)$. From
\(u^2-\widetilde u^2=(u-\widetilde u)(u+\widetilde u)\),
\begin{equation}
\label{eq:common-G-lipschitz}
    \|v-\widetilde v\|_V
    \le\beta_\rho\|u-\widetilde u\|_V.
\end{equation}
Using
\[
    \overline u\,v-\overline{\widetilde u}\,\widetilde v
    =\overline u\,(v-\widetilde v)
     +\overline{(u-\widetilde u)}\,\widetilde v,
\]
followed by \eqref{eq:common-map-bounds} and
\eqref{eq:common-G-lipschitz}, we obtain
\begin{align}
    \|\mathcal F_X(u)-\mathcal F_X(\widetilde u)\|_V
    &\le \widehat M_1C_{\rm prod}\|\chi_1\|_{L^\infty(B_R)}
       \left(\rho\|v-\widetilde v\|_V
       +\sigma_\rho\|u-\widetilde u\|_V\right)\notag\\
    &\le(\alpha_\rho\beta_\rho+\theta_\rho)
       \|u-\widetilde u\|_V
     =q_\rho\|u-\widetilde u\|_V.
    \label{eq:common-F-contraction}
\end{align}
By \eqref{eq:small-data-radius}, \(\mathcal F_X\) is a contraction on the
complete set \(\mathcal B_{\rho,X}^1\). Its unique fixed point $u_1$
in the continuous case, or $u_{1,h}$ in the discrete case, together with
$u_2=\mathcal G_V(u_1)$, or
$u_{2,h}=\mathcal G_{V_h}(u_{1,h})$, solves the corresponding coupled
problem. The first estimate in \eqref{eq:common-map-bounds} places the pair
in the ball in \eqref{eq:continuous-discrete-balls}. Conversely, every
solution in that ball gives a fixed point of \(\mathcal F_X\), which proves
uniqueness there.
\end{proof}

\begin{remark}[Continuous and discrete fixed-point rates]
\label{rem:fixed-point-rates}
Under the hypotheses of Theorem~\ref{thm:continuous-discrete-small-data},
let the continuous iterates be defined by
\[
    u_2^n=\mathcal G_V(u_1^n),
    \qquad
    u_1^{n+1}=\mathcal F_V(u_1^n),
\]
and let the discrete iterates be defined analogously using
$\mathcal G_{V_h}$ and $\mathcal F_{V_h}$. If
\(\|u_1^0\|_V\le\rho\), then \eqref{eq:common-map-bounds} and induction
give
\[
    \|u_1^n\|_V\le\rho,
    \qquad
    \|u_2^n\|_V\le\sigma_\rho
    \qquad\forall n\ge0,
\]
and \eqref{eq:common-F-contraction} and
\eqref{eq:common-G-lipschitz} give
\begin{equation}
\label{eq:continuous-fixed-point-rate}
    \|u_1-u_1^n\|_V+\|u_2-u_2^n\|_V
    \le(1+\beta_\rho)q_\rho^n\|u_1-u_1^0\|_V
\end{equation}
If \(\|u_{1,h}^0\|_V\le\rho\), the same argument gives
\[
    \|u_{1,h}^n\|_V\le\rho,
    \qquad
    \|u_{2,h}^n\|_V\le\sigma_\rho
    \qquad\forall n\ge0,
\]
and
\begin{equation}
\label{eq:discrete-fixed-point-rate}
    \|u_{1,h}-u_{1,h}^n\|_V+\|u_{2,h}-u_{2,h}^n\|_V
    \le(1+\beta_\rho)q_\rho^n
       \|u_{1,h}-u_{1,h}^0\|_V.
\end{equation}
In particular, if both initial bounds hold, then
\begin{equation}
\label{eq:uniform-iterate-bounds}
\begin{aligned}
    \|u_1^n\|_V,\ \|u_{1,h}^n\|_V&\le\rho,\\
    \|u_2^n\|_V,\ \|u_{2,h}^n\|_V&\le\sigma_\rho
    \qquad\forall n\ge0.
\end{aligned}
\end{equation}
\end{remark}

\subsection{Error estimates}
\label{subsec:fem-error-estimates}
We now derive error estimates in the same small-data regime in which the
continuous and discrete nonlinear problems are well posed. Throughout this
subsection, unless stated otherwise, we assume the hypotheses of
Theorem~\ref{thm:continuous-discrete-small-data} and fix a radius
\(\rho>0\) satisfying \eqref{eq:small-data-radius}. In particular,
\eqref{eq:continuous-discrete-balls} gives the uniform bounds
\begin{equation}
\label{eq:solution-radius-bounds}
    \|u_1\|_V,\ \|u_{1,h}\|_V\le\rho,
    \qquad
    \|u_2\|_V,\ \|u_{2,h}\|_V\le\sigma_\rho.
\end{equation}

The main result of the paper is the following a priori error estimate:
\begin{theorem}[A priori error estimate]
\label{thm:main-fem-error}
Let
$(u_1,u_2)\in\mathcal B_\rho$ and
$(u_{1,h},u_{2,h})\in\mathcal B_{\rho,h}$ be the solutions supplied by
Theorem~\ref{thm:continuous-discrete-small-data}. Then, for
$0<h\le h_0$,
\begin{equation}
\label{eq:a-priori-estimate}
    \begin{aligned}
    \|u_1-u_{1,h}\|_V+\|u_2-u_{2,h}\|_V
    \le
    C_{\rm FE}
    \left(
        \inf_{z_h\in V_h}\|u_1-z_h\|_V
        +
        \inf_{w_h\in V_h}\|u_2-w_h\|_V
    \right),
\end{aligned}
\end{equation}
where
\[
    C_{\rm FE}
    :=
    C_{\rm qo}
    \left[
        1
        +
        \frac{(1+\beta_\rho)(1+\alpha_\rho)}
        {1-q_\rho}
    \right].
\]
In particular, $C_{\rm FE}$ is independent of $h$.
\end{theorem}

\begin{proof}
Set
$e_1:=u_1-u_{1,h}$ and $e_2:=u_2-u_{2,h}$.
Define the best-approximation errors
$$
    \eta_{1,h}:=\inf_{z_h\in V_h}\|u_1-z_h\|_V,
    \qquad
    \eta_{2,h}:=\inf_{w_h\in V_h}\|u_2-w_h\|_V.
$$
Let $(\widetilde u_{1,h},\widetilde u_{2,h})\in V_h\times V_h$ be the
Galerkin solutions of the linear problems with the exact nonlinear data frozen:
\begin{align}
    a_1(\widetilde u_{1,h},\xi_h)
    &=
    \int_{B_R}
        \chi_1\overline{u_1}u_2\overline{\xi_h}\,dx
    +
    \ell(\xi_h)
    &\forall \xi_h\in V_h,
    \label{eq:frozen-linear-1}\\
    a_2(\widetilde u_{2,h},\eta_h)
    &=
    \int_{B_R}
        \chi_2u_1^2\overline{\eta_h}\,dx
    &\forall \eta_h\in V_h.
    \label{eq:frozen-linear-2}
\end{align}
Since $u_1$ and $u_2$ solve the corresponding continuous linear problems with
the same right-hand sides, Proposition~\ref{prop:linear-galerkin-stability} gives
\begin{equation}
     \|u_1-\widetilde u_{1,h}\|_V
    \le
    C_{\rm qo}\eta_{1,h},
    \qquad
    \|u_2-\widetilde u_{2,h}\|_V
    \le
    C_{\rm qo}\eta_{2,h}.
    \label{eq:frozen-qo}
\end{equation}
Subtracting \eqref{eq:discrete-shg-2} from \eqref{eq:frozen-linear-2}, we get
$$
    a_2(\widetilde u_{2,h}-u_{2,h},\eta_h)
    =
    \int_{B_R}
        \chi_2(u_1^2-u_{1,h}^2)\overline{\eta_h}\,dx
    \qquad\forall \eta_h\in V_h.
$$
By the discrete linear stability estimate \eqref{discrete_stab},
$$
\begin{aligned}
    \|\widetilde u_{2,h}-u_{2,h}\|_V
    &\le
    M_{\rm lin}\|R_h(\chi_2(u_1^2-u_{1,h}^2))\|_{V_h'}\\
    &\le
    \widehat M_2\|\chi_2(u_1^2-u_{1,h}^2)\|_{V'}.
\end{aligned}
$$
Since
$$
    u_1^2-u_{1,h}^2
    =
    (u_1-u_{1,h})(u_1+u_{1,h})
    =
    e_1(u_1+u_{1,h}),
$$
Lemma~\ref{lem:product-estimate-continuous} gives
$$
\begin{aligned}
    \|\chi_2(u_1^2-u_{1,h}^2)\|_{V'}
    &\le
    \|\chi_2\|_{L^\infty(B_R)}
    C_{\rm prod}
    \|e_1\|_V
    \|u_1+u_{1,h}\|_V\\
    &\le
    2C_{\rm prod}\|\chi_2\|_{L^\infty(B_R)}
    \rho
    \|e_1\|_V.
\end{aligned}
$$
Therefore
$$
    \|\widetilde u_{2,h}-u_{2,h}\|_V
    \le
    \beta_\rho\|e_1\|_V.
$$
Using the triangle inequality and \eqref{eq:frozen-qo},
\begin{equation}
        \|e_2\|_V
    \le
    C_{\rm qo}\eta_{2,h}
    +
    \beta_\rho\|e_1\|_V.
    \label{eq:e2-bound-final}
\end{equation}
Next, subtracting \eqref{eq:discrete-shg-1} from \eqref{eq:frozen-linear-1} and using the fact that $\ell_h(\xi_h)=\ell(\xi_h)$ gives
$$
    a_1(\widetilde u_{1,h}-u_{1,h},\xi_h)
    =
    \int_{B_R}
        \chi_1
        \left(
            \overline{u_1}u_2
            -
            \overline{u_{1,h}}u_{2,h}
        \right)
        \overline{\xi_h}\,dx
    \qquad\forall \xi_h\in V_h.
$$
By the discrete linear stability estimate \eqref{discrete_stab},
$$
\begin{aligned}
    \|\widetilde u_{1,h}-u_{1,h}\|_V
    &\le
    M_{\rm lin}
    \left\|R_h\left(
        \chi_1
        \left(
            \overline{u_1}u_2
            -
            \overline{u_{1,h}}u_{2,h}
        \right)
    \right)\right\|_{V_h'}\\
    &\le
    \widehat M_1
    \left\|
        \chi_1
        \left(
            \overline{u_1}u_2
            -
            \overline{u_{1,h}}u_{2,h}
        \right)
    \right\|_{V'}.
\end{aligned}
$$
We next decompose
$$
    \overline{u_1}u_2
    -
    \overline{u_{1,h}}u_{2,h}
    =
    \overline{u_1}(u_2-u_{2,h})
    +
    \overline{(u_1-u_{1,h})}u_{2,h}
    =
    \overline{u_1}e_2+\overline{e_1}u_{2,h}.
$$
Applying Lemma~\ref{lem:product-estimate-continuous} and
\eqref{eq:solution-radius-bounds},
$$
\begin{aligned}
    \left\|
        \chi_1
        \left(
            \overline{u_1}u_2
            -
            \overline{u_{1,h}}u_{2,h}
        \right)
    \right\|_{V'}\le
    C_{\rm prod}\|\chi_1\|_{L^\infty(B_R)}
    \left(
        \rho\|e_2\|_V+\sigma_\rho\|e_1\|_V
    \right).
\end{aligned}
$$
Thus
$$
    \|\widetilde u_{1,h}-u_{1,h}\|_V
    \le
    \alpha_\rho\|e_2\|_V+\theta_\rho\|e_1\|_V.
$$
Again using the triangle inequality and \eqref{eq:frozen-qo},
\begin{equation}
        \|e_1\|_V
    \le
    C_{\rm qo}\eta_{1,h}
    +
    \alpha_\rho\|e_2\|_V+\theta_\rho\|e_1\|_V.
    \label{eq:e1-pre-final}
\end{equation}
Substituting \eqref{eq:e2-bound-final} into \eqref{eq:e1-pre-final}, we find
$$
    \|e_1\|_V
    \le
    C_{\rm qo}\eta_{1,h}
    +
    \theta_\rho\|e_1\|_V
    +
    \alpha_\rho
    \left(
        C_{\rm qo}\eta_{2,h}
        +
        \beta_\rho\|e_1\|_V
    \right).
$$
Hence
$$
    (1-q_\rho)\|e_1\|_V
    \le
    C_{\rm qo}\eta_{1,h}
    +
    \alpha_\rho C_{\rm qo}\eta_{2,h}.
$$
Since $q_\rho<1$ by \eqref{eq:small-data-radius},
\begin{equation}
        \|e_1\|_V
    \le
    \frac{C_{\rm qo}}{1-q_\rho}
    \left(
        \eta_{1,h}+\alpha_\rho\eta_{2,h}\right)\le
    \frac{C_{\rm qo}(1+\alpha_\rho)}{1-q_\rho}
    (\eta_{1,h}+\eta_{2,h}).
    \label{eq:e1-final-main}
\end{equation}
Using \eqref{eq:e2-bound-final} and \eqref{eq:e1-final-main},
\begin{equation}
\label{eq:e2-final-main}
    \begin{aligned}
    \|e_2\|_V
    \le
    C_{\rm qo}
    \left(
        1+\frac{\beta_\rho(1+\alpha_\rho)}{1-q_\rho}
    \right)
    (\eta_{1,h}+\eta_{2,h}).
\end{aligned}
\end{equation}
Adding the estimates \eqref{eq:e1-final-main}-\eqref{eq:e2-final-main} proves \eqref{eq:a-priori-estimate}.
\end{proof}
The previous result then directly implies an explicit error rate with respect to the mesh size.
\begin{corollary}[Convergence rate]
\label{cor:fem-convergence-rate}
Assume the hypotheses of Theorem~\ref{thm:main-fem-error}. If
Assumption~\ref{ass:fe-approx} holds and
$u_1\in H^{1+s_1}(B_R)$, $u_2\in H^{1+s_2}(B_R)$, for $s_1,s_2\in(0,p]$,
then for $0<h<h_0$,
$$
\begin{aligned}
    \|u_1-u_{1,h}\|_{V}
    +
    \|u_2-u_{2,h}\|_{V}
    \le
    C
    \left(
        h^{s_1}|u_1|_{H^{1+s_1}(B_R)}
        +
        h^{s_2}|u_2|_{H^{1+s_2}(B_R)}
    \right),
\end{aligned}
$$
where $C$ is independent of $h$. In particular, if
$u_1,u_2\in H^{1+s}(B_R)$ for a common $s\in(0,p]$, then
$$
    \|u_1-u_{1,h}\|_{V}
    +
    \|u_2-u_{2,h}\|_{V}
    \le
    Ch^s
    \left(
        |u_1|_{H^{1+s}(B_R)}
        +
        |u_2|_{H^{1+s}(B_R)}
    \right).
$$
\end{corollary}
The previous theorem compares the exact nonlinear solution with the exact
discrete nonlinear fixed-point. In practice, one may compute the discrete
solution by iterating the fixed-point map. We therefore record the analogue of
the finite element error estimate for the fixed-point iterates and analyze how the error propagates with each iteration. Define the continuous iterates by choosing $u_1^0\in V$ and, for $n\ge0$,
solving
\begin{align}
    a_2(u_2^n,\eta)
    &=
    \int_{B_R}
        \chi_2(u_1^n)^2\overline{\eta}\,dx
    &&\forall \eta\in V,
    \label{eq:continuous-iterate-2}\\
    a_1(u_1^{n+1},\xi)
    &=
    \int_{B_R}
        \chi_1\overline{u_1^n}u_2^n\overline{\xi}\,dx
    +
    \ell(\xi)
    &&\forall \xi\in V,
    \label{eq:continuous-iterate-1}
\end{align}
for $(u_1^{n+1},u_2^n)\in V\times V$.
Similarly, given $u_{1,h}^0\in V_h$, define the discrete iterates $(u_{1,h}^{n+1},u_{2,h}^n)$, $n\ge 0$, by
\begin{align}
    a_2(u_{2,h}^n,\eta_h)
    &=
    \int_{B_R}
        \chi_2(u_{1,h}^n)^2\overline{\eta_h}\,dx
    &&\forall \eta_h\in V_h,
    \label{eq:discrete-iterate-2}\\
    a_1(u_{1,h}^{n+1},\xi_h)
    &=
    \int_{B_R}
        \chi_1\overline{u_{1,h}^n}u_{2,h}^n\overline{\xi_h}\,dx
    +
    \ell(\xi_h)
    &&\forall \xi_h\in V_h.
    \label{eq:discrete-iterate-1}
\end{align}
\begin{theorem}[Fixed-point estimate]
\label{thm:iterate-error}
Assume the hypotheses of
Theorem~\ref{thm:continuous-discrete-small-data}, let $0<h\le h_0$, and
let the continuous and discrete fixed-point iterates be defined by
\eqref{eq:continuous-iterate-2}-\eqref{eq:continuous-iterate-1} and
\eqref{eq:discrete-iterate-2}-\eqref{eq:discrete-iterate-1}, respectively.
Suppose that
\begin{equation}
\label{eq:iterate-initial-bounds}
    \|u_1^0\|_V\le\rho,
    \qquad
    \|u_{1,h}^0\|_V\le\rho.
\end{equation}
Then the bounds in \eqref{eq:uniform-iterate-bounds} hold for every
$n\ge0$. Set
\begin{equation*}
     \eta_{1,h}^{n+1}
    :=
    \inf_{z_h\in V_h}\|u_1^{n+1}-z_h\|_V,
    \qquad
    \eta_{2,h}^{n}
    :=
    \inf_{w_h\in V_h}\|u_2^{n}-w_h\|_V.
\end{equation*}
Then
\begin{equation}
        \|u_2^n-u_{2,h}^n\|_V
    \le
    C_{\rm qo}\eta_{2,h}^{n}
    +
    \beta_\rho\|u_1^n-u_{1,h}^n\|_V,
    \label{eq:iterate-u2-estimate}
\end{equation}
and
\begin{equation}
     \|u_1^{n+1}-u_{1,h}^{n+1}\|_V
    \le
    C_{\rm qo}\eta_{1,h}^{n+1}
    +
    \theta_\rho\|u_1^n-u_{1,h}^n\|_V
    +
    \alpha_\rho\|u_2^n-u_{2,h}^n\|_V.
    \label{eq:iterate-u1-estimate}
\end{equation}
For every $n\ge0$,
\begin{equation}
\label{eq:total_geom_series_estimate}
\begin{aligned}
    \|u_1^n-u_{1,h}^n\|_V&+\|u_2^n-u_{2,h}^n\|_V\\
    &\le(1+\beta_\rho)q_\rho^n\|u_1^0-u_{1,h}^0\|_V+C_{\rm qo}(1+\beta_\rho)
        \sum_{k=1}^{n}q_\rho^{n-k}\eta_{1,h}^{k}\\
        &\qquad\qquad\qquad\qquad\qquad\qquad\qquad+C_{\rm qo}(1+\beta_\rho)\alpha_\rho
        \sum_{m=0}^{n-1}q_\rho^{n-1-m}\eta_{2,h}^{m}
        +C_{\rm qo}\eta_{2,h}^{n},
\end{aligned}
\end{equation}
where a sum with an empty index set is understood to be zero.
\end{theorem}

\begin{proof}
This proof repeats the frozen-data argument of
Theorem~\ref{thm:main-fem-error} at a fixed iterate. First let $(\widetilde{u}_{1,h}^{\,n+1}, \widetilde{u}_{2,h}^{\,n})\in V_h\times V_h$ solve the discrete linear problem with
the exact iterate data:
\begin{align}
    a_1(\widetilde u_{1,h}^{\,n+1},\xi_h)
    &=
    \int_{B_R}
        \chi_1\overline{u_1^n}u_2^n\overline{\xi_h}\,dx
    +
    \ell(\xi_h)
    &\forall \xi_h\in V_h,
    \label{eq:frozen-iterate-1}\\
    a_2(\widetilde u_{2,h}^{\, n},\eta_h)
    &=
    \int_{B_R}
        \chi_2(u_1^n)^2\overline{\eta_h}\,dx
    &\forall \eta_h\in V_h.
    \label{eq:frozen-iterate-2}
\end{align}
Proposition~\ref{prop:linear-galerkin-stability} gives
\[
    \|u_1^{n+1}-\widetilde u_{1,h}^{\,n+1}\|_V
    \le C_{\rm qo}\eta_{1,h}^{n+1},
    \qquad
    \|u_2^n-\widetilde u_{2,h}^{\,n}\|_V
    \le C_{\rm qo}\eta_{2,h}^{n}.
\]
Subtracting \eqref{eq:discrete-iterate-1} from \eqref{eq:frozen-iterate-1} gives
\begin{align*} 
a_1(\widetilde u_{1,h}^{\,n+1}-u_{1,h}^{\,n+1},\xi_h)
    &=
    \int_{B_R}
        \chi_1
        \left(
            \overline{u_1^n}u_2^n
            -
            \overline{u_{1,h}^n}u_{2,h}^n
        \right)
        \overline{\xi_h}\,dx\\
        &= \int_{B_R}
        \chi_1
        \left(
    \overline{u_1^n}(u_2^n-u_{2,h}^n)
    +
    \overline{(u_1^n-u_{1,h}^n)}u_{2,h}^n
        \right)
        \overline{\xi_h}\,dx,
\end{align*}
and subtracting \eqref{eq:discrete-iterate-2} from \eqref{eq:frozen-iterate-2} gives
\begin{align*}
    a_2(\widetilde u_{2,h}^{\, n}-u_{2,h}^n,\eta_h)
    &=
    \int_{B_R}
        \chi_2\left((u_1^n)^2-(u_{1,h}^n)^2\right)
        \overline{\eta_h}\,dx.
\end{align*}
Using \eqref{discrete_stab} and Lemma~\ref{lem:product-estimate-continuous} we obtain
$$
\begin{aligned}
    \|\widetilde u_{1,h}^{\,n+1}-u_{1,h}^{n+1}\|_V
    &\le
    \theta_\rho\|u_1^n-u_{1,h}^n\|_V
    +\alpha_\rho\|u_2^n-u_{2,h}^n\|_V,\\
    \|\widetilde u_{2,h}^{\, n}-u_{2,h}^n\|_V
    &\le
    \beta_\rho\|u_1^n-u_{1,h}^n\|_V.
\end{aligned}
$$
The triangle inequality proves \eqref{eq:iterate-u2-estimate} and \eqref{eq:iterate-u1-estimate}. Substituting
\eqref{eq:iterate-u2-estimate} into \eqref{eq:iterate-u1-estimate} gives
\begin{equation}
        \|u_1^{n+1}-u_{1,h}^{n+1}\|_V
    \le
    C_{\rm qo}\eta_{1,h}^{n+1}
    +
    \alpha_\rho C_{\rm qo}\eta_{2,h}^{n}
    +
    q_\rho\|u_1^n-u_{1,h}^n\|_V.
    \label{eq:iterate-recursion}
\end{equation} Iterating the scalar recurrence
\eqref{eq:iterate-recursion} yields the geometric-series estimate:
\begin{equation}
\label{eq:geometric-series}
        \|u_1^n-u_{1,h}^n\|_V
    \le
    q_\rho^n\|u_1^0-u_{1,h}^0\|_V
    +
    C_{\rm qo}\sum_{m=0}^{n-1}
    q_\rho^{\,n-1-m}
    \left(
        \eta_{1,h}^{m+1}
        +
        \alpha_\rho \eta_{2,h}^{m}
    \right).
\end{equation}
Finally, combining \eqref{eq:geometric-series} with \eqref{eq:iterate-u2-estimate} results in \eqref{eq:total_geom_series_estimate}:
\[\begin{aligned}
    &\|u_1^n-u_{1,h}^n\|_V+\|u_2^n-u_{2,h}^n\|_V\\
    &\le (1+\beta_\rho)\|u_1^n-u_{1,h}^n\|_V+C_{\rm qo}\eta_{2,h}^n \\
    &\le (1+\beta_\rho)q_\rho^n\|u_1^0-u_{1,h}^0\|_V
    +C_{\rm qo}(1+\beta_\rho)\sum_{m=0}^{n-1}
    q_\rho^{\,n-1-m}
    \left(
        \eta_{1,h}^{m+1}
        +
        \alpha_\rho \eta_{2,h}^{m}
    \right)+C_{\rm qo}\eta_{2,h}^n\\
    &=(1+\beta_\rho)q_\rho^n\|u_1^0-u_{1,h}^0\|_V
    +C_{\rm qo}(1+\beta_\rho)\sum_{k=1}^{n}q_\rho^{n-k}\eta_{1,h}^{k}+C_{\rm qo}(1+\beta_\rho)\alpha_\rho
    \sum_{m=0}^{n-1}q_\rho^{n-1-m}\eta_{2,h}^{m}
    +C_{\rm qo}\eta_{2,h}^{n}.
\end{aligned}\]
\end{proof}
Our next result provides an explicit rate of the error between the continuous and discrete fixed-point iterates.
\begin{corollary}[Fixed-point error rate]
\label{cor:iterate-rate}
Assume the hypotheses of Theorem~\ref{thm:iterate-error} and 
Assumption~\ref{ass:fe-approx}. Let $s\in(0,p]$. Suppose there exists
$C_s>0$ such that, for every $n\ge0$,
\begin{equation}
        u_1^n,u_2^n\in H^{1+s}(B_R),
    \qquad
    |u_1^n|_{H^{1+s}(B_R)}+|u_2^n|_{H^{1+s}(B_R)}
    \le C_s.
\end{equation}
Then, for every $n\ge0$,
\begin{equation}
\label{eq:total-h-bound}
    \|u_1^n-u_{1,h}^n\|_V+
    \|u_2^n-u_{2,h}^n\|_V
    \le (1+\beta_\rho)q_\rho^n\|u_1^0-u_{1,h}^0\|_V
    +
    C_{\rm qo}C_{\rm app}C_s
    \left[
        1+\frac{(1+\beta_\rho)(1+\alpha_\rho)}{1-q_\rho}
    \right]h^s.
\end{equation}
\end{corollary}

\begin{proof}
By Assumption~\ref{ass:fe-approx} and the uniform bound by $C_s$,
$$
    \eta_{1,h}^{m+1}
    \le
    C_{\rm app}h^s|u_1^{m+1}|_{H^{1+s}(B_R)}\le C_{\rm app}C_sh^s,
    \qquad
    \eta_{2,h}^{m}
    \le
    C_{\rm app}h^s|u_2^{m}|_{H^{1+s}(B_R)}\le C_{\rm app}C_sh^s.
$$
Therefore
$$
    C_{\rm qo}\eta_{1,h}^{m+1}
    +
    \alpha_\rho C_{\rm qo}\eta_{2,h}^{m}
    \le
    C_{\rm qo}C_{\rm app}(1+\alpha_\rho)C_s h^s.
$$
Substituting this into \eqref{eq:geometric-series} gives
\begin{equation}
\label{eq:h-bound-1}
    \begin{aligned}
    \|u_1^n-u_{1,h}^n\|_V
    &\le
    q_\rho^n\|u_1^0-u_{1,h}^0\|_V
    +
    \frac{C_{\rm qo}C_{\rm app}C_s(1+\alpha_\rho)}{1-q_\rho} h^s.
\end{aligned}
\end{equation}
Moreover
\begin{equation}
    \label{eq:h-bound-2}
    \|u_2^n-u_{2,h}^n\|_V
    \le
    C_{\rm qo}C_{\rm app}C_s h^s+\beta_\rho\left(q_\rho^n\|u_1^0-u_{1,h}^0\|_V
    +
    \frac{C_{\rm qo}C_{\rm app}C_s(1+\alpha_\rho)}{1-q_\rho} h^s\right).
\end{equation}
Combining \eqref{eq:h-bound-1}-\eqref{eq:h-bound-2} and simplifying, we obtain \eqref{eq:total-h-bound}.
\end{proof}
\begin{remark}
\label{rem:iterate-rate-s-one}
For \(s=1\), the uniform regularity assumption in
Corollary~\ref{cor:iterate-rate} need not be imposed separately. More
precisely, under the hypotheses of Theorem~\ref{thm:iterate-error}, there
exists a constant \(C_{1}>0\), independent of the iteration \(m\), such
that
\begin{equation}
\label{eq:uniform-iterate-H2-bound}
    \sup_{m\ge0}
    \left(
        \|u_1^{m+1}\|_{H^2(B_R)}
        +
        \|u_2^m\|_{H^2(B_R)}
    \right)
    \le C_{1}.
\end{equation}
Consequently, if \(p\ge1\), estimate \eqref{eq:total-h-bound} holds with $s=1$ and
$C_s=C_1$. Notice that no \(H^2(B_R)\)-regularity
of the initial iterate \(u_1^0\) is required.

To prove \eqref{eq:uniform-iterate-H2-bound}, fix \(R'>R\), and, for
\(j\in\{1,2\}\), let \(w_{j,g}\) denote the outgoing solution of
\eqref{wg1}-\eqref{wg-bc} with \(k=\kappa_j\) and Dirichlet datum \(g\).
Define extensions to \(B_{R'}\) by
\[
    U_2^m(x)
    :=
    \begin{cases}
        u_2^m(x),
            &x\in B_R,\\
        w_{2,\gamma u_2^m}(x),
            &x\in B_{R'}\setminus\overline{B_R},
    \end{cases}
\]
and
\[
    U_1^{m+1}(x)
    :=
    \begin{cases}
        u_1^{m+1}(x),
            &x\in B_R,\\
        u^i(x)+
        w_{1,\gamma(u_1^{m+1}-u^i)}(x),
            &x\in B_{R'}\setminus\overline{B_R}.
    \end{cases}
\]
Let \(\widetilde n_j\) denote the extension of \(n_j\) to \(B_{R'}\)
obtained by setting \(\widetilde n_j=1\) on
\(B_{R'}\setminus B_R\).
The trace continuity in these definitions and the exact DtN conditions
\eqref{eq:u1-dtn-bc}-\eqref{eq:u2-dtn-bc} imply that the traces of the
normal derivatives also agree on \(\Gamma_R\). Hence
\(U_1^{m+1},U_2^m\in H^1(B_{R'})\), and the boundary distributions
arising at \(\Gamma_R\) cancel. In particular, these extensions satisfy
\begin{align}
    \Delta U_1^{m+1}
    +\kappa_1^2\widetilde n_1 U_1^{m+1}
    &=\begin{cases}
        -\chi_1 \overline{u_1^m}u_2^m,&x\in B_R, \\
        0,& x\in B_{R'}\setminus B_R,
    \end{cases}
    \label{eq:extended-iterate-u1}\\
    \Delta U_2^m
    +\kappa_2^2\widetilde n_2 U_2^m
    &=\begin{cases}
        -\chi_2(u_1^m)^2,&x\in B_R, \\
        0,& x\in B_{R'}\setminus B_R,
    \end{cases}
    \label{eq:extended-iterate-u2}
\end{align}
in the distributional sense in \(B_{R'}\). The boundedness of the exterior Dirichlet solution operator and the
trace theorem give
\begin{equation}
\label{eq:extended-iterate-H1-bound}
    \|U_1^{m+1}\|_{H^1(B_{R'})}
    +
    \|U_2^m\|_{H^1(B_{R'})}
    \le
    C\left(
        \rho+\sigma_\rho+
        \|u^i\|_{H^1(B_{R'})}
    \right),
\end{equation}
where we used \eqref{eq:uniform-iterate-bounds}. Moreover, since
\(H^1(B_R)\hookrightarrow L^4(B_R)\) for \(d\in\{2,3\}\),
\begin{align}
    \|\chi_2(u_1^m)^2\|_{L^2(B_R)}
    &\le
    C\|\chi_2\|_{L^\infty(B_R)}\rho^2,
    \label{eq:iterate-source-L2-u2}\\
    \|\chi_1\overline{u_1^m}u_2^m\|_{L^2(B_R)}
    &\le
    C\|\chi_1\|_{L^\infty(B_R)}
    \rho\sigma_\rho.
    \label{eq:iterate-source-L2-u1}
\end{align}
It follows from \eqref{eq:extended-iterate-u1}-\eqref{eq:iterate-source-L2-u1}
that \(\Delta U_1^{m+1},\Delta U_2^m\in L^2(B_{R'})\), uniformly in
\(m\). Since \(B_R\Subset B_{R'}\), interior elliptic regularity for the
Laplacian therefore yields
\[
\begin{aligned}
    \|u_1^{m+1}\|_{H^2(B_R)}
    +
    \|u_2^m\|_{H^2(B_R)}
    \le C\Big(
        &\rho+\sigma_\rho+\|u^i\|_{H^1(B_{R'})}+
        \|\chi_1\|_{L^\infty(B_R)}\rho\sigma_\rho
        +
        \|\chi_2\|_{L^\infty(B_R)}\rho^2
    \Big),
\end{aligned}
\]
which proves \eqref{eq:uniform-iterate-H2-bound}. Finally, the proof of Corollary~\ref{cor:iterate-rate} uses only the
approximation errors \(\eta_{1,h}^{m+1}\) and \(\eta_{2,h}^m\).
Therefore \eqref{eq:uniform-iterate-H2-bound}, rather than regularity of
\(u_1^0\), is the uniform regularity needed when \(s=1\).
\end{remark}
We conclude with two representations of the total error.
\begin{corollary}[Total error decompositions]
\label{cor:total-error-decompositions}
Assume the hypotheses of
Theorem~\ref{thm:continuous-discrete-small-data}, and let
$\{(u_{1,h}^n,u_{2,h}^n)\}_{n\ge0}$ be the discrete iterates generated by
\eqref{eq:discrete-iterate-2}-\eqref{eq:discrete-iterate-1} with
$\|u_{1,h}^0\|_V\le\rho$.

\begin{enumerate}
\item If the hypotheses of Corollary~\ref{cor:fem-convergence-rate} hold with
$s_1=s_2=s$, then, for every $n\ge0$,
\begin{equation}
\label{eq:total-error-discrete-fixed-point}
\begin{aligned}
    \|u_1-u_{1,h}^n\|_V&+\|u_2-u_{2,h}^n\|_V\\
    &\le
    C h^s
    \left(
        |u_1|_{H^{1+s}(B_R)}
        +
        |u_2|_{H^{1+s}(B_R)}
    \right)+
    (1+\beta_\rho)q_\rho^n\|u_{1,h}-u_{1,h}^0\|_V,
\end{aligned}
\end{equation}
where $C$ is independent of $h$ and $n$.

\item Let $\{(u_1^n,u_2^n)\}_{n\ge0}$ be the continuous iterates generated by
\eqref{eq:continuous-iterate-2}-\eqref{eq:continuous-iterate-1} with
$\|u_1^0\|_V\le\rho$. If the hypotheses of
Corollary~\ref{cor:iterate-rate} also hold, then, for every $n\ge0$,
\begin{equation}
\label{eq:total-error-continuous-discrete-iterates}
\begin{aligned}
    &\|u_1-u_{1,h}^n\|_V+\|u_2-u_{2,h}^n\|_V
    \\ &\le
    (1+\beta_\rho)q_\rho^n\left(\|u_1-u_1^0\|_V+\|u_1^0-u_{1,h}^0\|_V\right)+C_{\rm qo}C_{\rm app}C_s
    \left[
        1+\frac{(1+\beta_\rho)(1+\alpha_\rho)}{1-q_\rho}
    \right]h^s.
\end{aligned}
\end{equation}
\end{enumerate}
\end{corollary}

\begin{proof}
For \eqref{eq:total-error-discrete-fixed-point}, the triangle inequality gives
\[
\begin{aligned}
    \|u_1-u_{1,h}^n\|_V+\|u_2-u_{2,h}^n\|_V
    &\le
    \|u_1-u_{1,h}\|_V+\|u_2-u_{2,h}\|_V+
    \|u_{1,h}-u_{1,h}^n\|_V+\|u_{2,h}-u_{2,h}^n\|_V.
\end{aligned}
\]
The first two terms are controlled by Corollary~\ref{cor:fem-convergence-rate},
and the last two terms are controlled by \eqref{eq:discrete-fixed-point-rate}.
This proves \eqref{eq:total-error-discrete-fixed-point}. For \eqref{eq:total-error-continuous-discrete-iterates}, the triangle inequality gives
\[
\begin{aligned}
    \|u_1-u_{1,h}^n\|_V+\|u_2-u_{2,h}^n\|_V
    &\le
    \|u_1-u_1^n\|_V+\|u_2-u_2^n\|_V+
    \|u_1^n-u_{1,h}^n\|_V+\|u_2^n-u_{2,h}^n\|_V.
\end{aligned}
\]
The first two terms are controlled by \eqref{eq:continuous-fixed-point-rate},
and the last two terms are controlled by \eqref{eq:total-h-bound}. This proves
\eqref{eq:total-error-continuous-discrete-iterates}.
\end{proof}

The estimate \eqref{eq:total-error-discrete-fixed-point} separates the finite element
error from the contraction error of the discrete nonlinear solver, while
\eqref{eq:total-error-continuous-discrete-iterates} separates the continuous fixed-point
error from the finite element error propagated between corresponding iterates. We further observe that a sufficient condition for $O(h^s)$ convergence is to take $n \gtrsim \lceil \frac{s\log h}{\log q_\rho}\rceil$.

\section{Numerical results}
\label{sec:numerics}

We present numerical results in dimensions \(d=2,3\), computed using version 6.2.2506-216-g588357f34 of the
NGSolve software package~\cite{netgen} on a Mac-Studio M1 Ultra with 128GB of RAM. In all experiments, curved
boundaries are represented using NGSolve's blending-function approach.
The geometry is approximated using polynomials of degree \(p\), matching
the polynomial degree of the finite element space in the volume.  In all cases, the iterative
fixed point solver is used with initial guess \(u_{1,h}^0=0\).

\subsection{Convergence rates}
The DtN boundary condition used to truncate the physical domain makes it
difficult to construct a convenient analytic solution. We therefore
replace it with a first-order absorbing boundary condition, which
provides a straightforward test of both the finite element discretization
and the nonlinear solver. We also set \(n_1=n_2=1\) and allow the
nonlinearity to act throughout \(B_R\). More precisely, we solve
\begin{align}
    \Delta u_1+\kappa_1^2 u_1
    &=
    -\chi_1 \overline{u_1}u_2+f
    &&\text{in }B_R,
    \label{eq:full-shg-u1-red}\\
    \Delta u_2+\kappa_2^2 u_2
    &=
    -\chi_2 u_1^2+g
    &&\text{in }B_R,\\
    \frac{\partial u_1}{\partial r}-i\kappa_1u_1
    &=
    f_{1,b}
    &&\text{on }\Gamma_R,\\
    \frac{\partial u_2}{\partial r}-i\kappa_2u_2
    &=
    f_{2,b}
    &&\text{on }\Gamma_R.
    \label{eq:imp-red}
\end{align}
Here \(f\), \(g\), \(f_{1,b}\), and \(f_{2,b}\) are prescribed source
terms. The finite element approximation of this problem can be analyzed
in the same manner as the DtN-truncated scattering problem
\eqref{eq:u1strong}-\eqref{eq:u2-dtn-bc}, with the DtN operator
\(T_\kappa\) replaced by multiplication by \(i\kappa\). The corresponding
linear stability estimates are studied by Esterhazy and
Melenk~\cite{EsterhazyMelenk}.
\begin{figure}[t]
    \begin{center}
    \begin{tabular}{cc}
        \resizebox{0.4\textwidth}{!}{\includegraphics{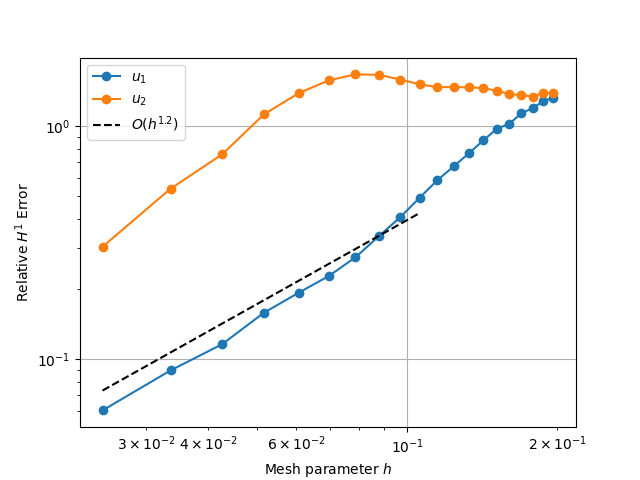}}&
       \resizebox{0.4\textwidth}{!}{\includegraphics{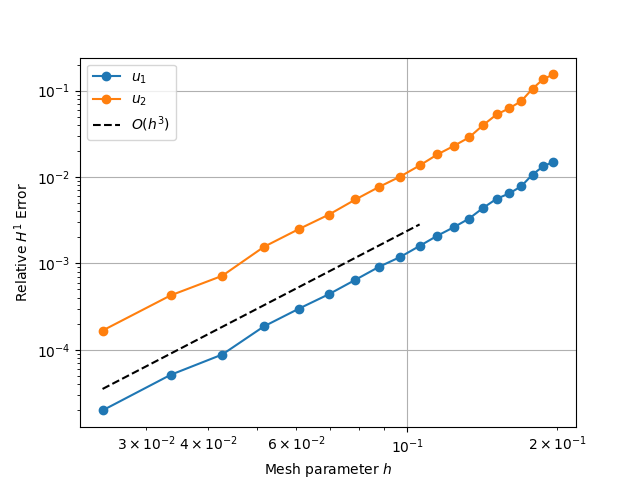}}\\
        a) $p=1$&b) $p=3$
            \end{tabular}
    \end{center}
\caption{Convergence curves for \(u_1\) and \(u_2\) using polynomial
    degrees \(p=1\) (left) and \(p=3\) (right).}
\label{fig2}
\end{figure}
To construct a smooth manufactured solution, we choose
\[
    u_1(x,y)=\exp(i\alpha x),
    \qquad
    u_2(x,y)=\exp(i\beta y).
\]
Substitution into \eqref{eq:full-shg-u1-red}-\eqref{eq:imp-red} gives
\[
    f
    =
    (\kappa_1^2-\alpha^2)\exp(i\alpha x)
    +\chi_1\exp(-i\alpha x)\exp(i\beta y),
\]
and
\[
    g
    =
    (\kappa_2^2-\beta^2)\exp(i\beta y)
    +\chi_2\exp(2i\alpha x).
\]
The corresponding boundary data are
\begin{figure}[t]
    \begin{center}
    \begin{tabular}{cc}
        \resizebox{0.45\textwidth}{!}{\includegraphics{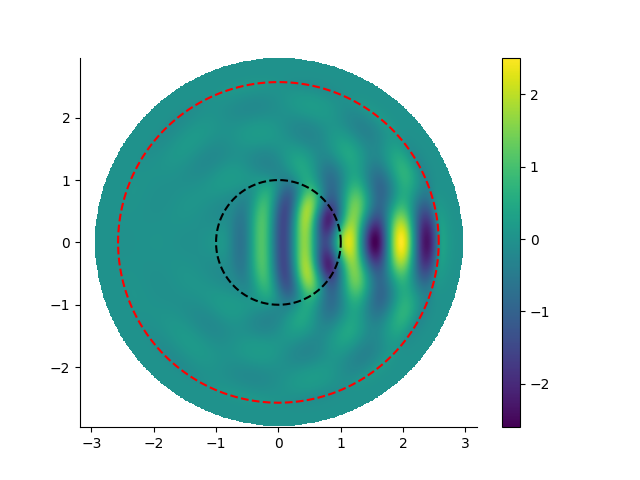}}&
       \resizebox{0.45\textwidth}{!}{\includegraphics{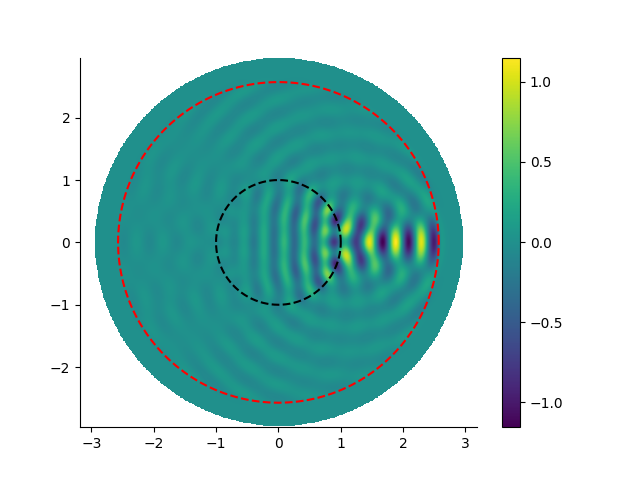}}\\
            \end{tabular}
    \end{center}
\caption{Fields computed using $p=3$ for the scattering problem with PML. We show the real part of the fundamental mode $u_{1,h}^s$ (left) and secondary mode $u_{2,h}$  (right). In each plot the red dashed curve indicates the start of the PML, while the black dashed curve is the
boundary of the scatterer $\partial D$.}
\label{fig3}
\end{figure}
\[
    f_{1,b}
    =
    i(\alpha\nu_1-\kappa_1)\exp(i\alpha x),
    \qquad
    f_{2,b}
    =
    i(\beta\nu_2-\kappa_2)\exp(i\beta y),
\]
where \(\nu=(\nu_1,\nu_2)\) denotes the outward unit normal on
\(\Gamma_R=\partial B_R\). Thus, the prescribed functions \(u_1\) and
\(u_2\) solve \eqref{eq:full-shg-u1-red}-\eqref{eq:imp-red} exactly.
For the experiments reported below, we take
\[
    \kappa_1=8, 
    \qquad
    \kappa_2=16,
    \qquad
    \alpha=7.8,
    \qquad
    \beta=15.8,
    \qquad
    \chi_1=8,
    \qquad
    \chi_2=10.
\]
The computational domain is the disc of radius
\(  R=1+2.5\left(\frac{2\pi}{\kappa_1}\right),
\)
and the nonlinearity is supported throughout the disc. The manufactured
problem includes the independent source $g$ in the second equation, so
Theorem~\ref{thm:continuous-discrete-small-data} and
\eqref{eq:coefficient-smallness} do not apply verbatim. Moreover, the
linear stability and product constants are not available quantitatively,
so an analogous smallness condition cannot be verified a priori for these
parameter values. Nevertheless,
although the manufactured solution is not intended to model a physical
scattering configuration, its finite element approximation still
requires the solution of the fully nonlinear problem. It therefore
provides a useful test of the predicted convergence rates and of the
nonlinear fixed-point solver.
We use 20 requested mesh sizes \(h\), equally spaced over the interval
\(
    \left[\frac{\lambda_2}{16},\frac{\lambda_2}{2}\right]\)
where \(\lambda_2=\frac{2\pi}{\kappa_2}\) is the wavelength of the secondary field.

The fixed-point iteration is terminated when the degree-of-freedom-normalized
\(\ell^2\)-norm of the change in the nodal values of \(u_{1,h}\) between
successive iterates falls below \(10^{-6}\). Since the fixed-point
analysis is formulated in terms of successive updates of \(u_1\), it is
sufficient to monitor \(u_{1,h}\) alone. For \(p=1\), the method requires
25 iterations on the coarsest mesh and 22 iterations on the finest mesh.
When $p=3$, convergence occurred in $22$ iterations on the coarsest mesh and 19 iterations on the finest mesh.
This weak dependence of the iteration count on \(h\) is consistent with
Remark~\ref{rem:fixed-point-rates}, once the mesh is sufficiently fine. The convergence results for \(p=1\) and \(p=3\) are shown in
Fig.~\ref{fig2}. For \(p=1\), the approximation \(u_{2,h}\) remains
inaccurate on the coarser meshes because they do not adequately resolve
a wave with wavenumber \(\kappa_2\). On the finer meshes, the observed
rate is approximately \(O(h^{1.2})\), which is close to the predicted
\(O(h)\) rate. The slight discrepancy is likely explained by the fact
that \(u_{2,h}\) has not yet fully entered the asymptotic regime, even on
the finest mesh. For \(p=3\), the observed cubic convergence agrees with
the theoretical prediction because even the initial mesh adequately
resolves both the fundamental and secondary fields. These experiments
therefore confirm both the predicted finite element convergence rates
and the necessity of the mesh-resolution assumption in the analysis.

\begin{figure}[t]
    \begin{center}
    \begin{tabular}{cc}
        \resizebox{0.45\textwidth}{!}{\includegraphics{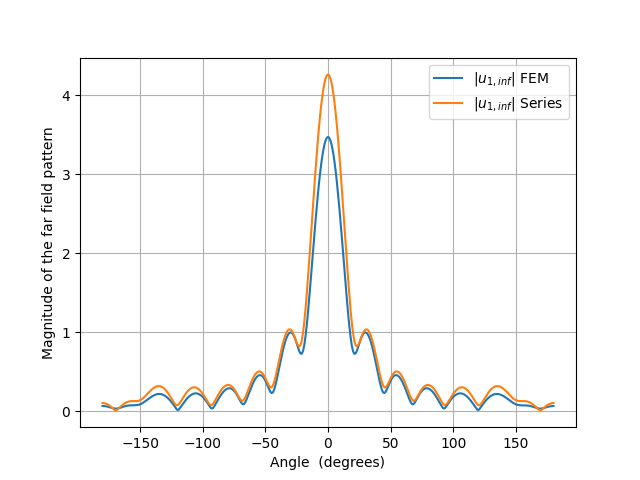}}&
       \resizebox{0.45\textwidth}{!}{\includegraphics{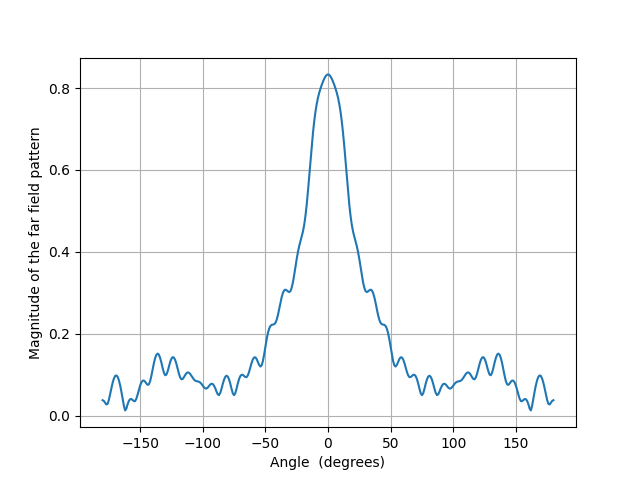}}\\
            \end{tabular}
    \end{center}
\caption{Moduli of the far-field patterns associated with the
    fields in Fig.~\ref{fig3}, plotted as functions of the observation
    angle: \(\lvert u_{1,h,\infty}^s\rvert\) for the fundamental field
    (left) and \(\lvert u_{2,h,\infty}\rvert\) for the secondary field
    (right). For comparison, the left panel also shows the fundamental
    far-field pattern for the corresponding linear problem
    \(\chi_1=\chi_2=0\), computed using a series solution.}
\label{fig4}
\end{figure}

\subsection{Discretization of the scattering problem}
We next compute the scattered fundamental field \(u_1^s\) and the
secondary field \(u_2\). Substituting the decomposition
\eqref{u1tot} into \eqref{eq:full-shg-u1} gives
\begin{align}
    \Delta u_1^s+\kappa_1^2 n_1u_1^s
    &=
    -\chi_1\overline{(u_1^s+u^i)}u_2
    +\kappa_1^2(1-n_1)u^i
    &&\text{in }B_R,
    \label{eq:full-shg-u1-pml}\\
    \Delta u_2+\kappa_2^2 n_2u_2
    &=
    -\chi_2(u_1^s+u^i)^2
    &&\text{in }B_R.
    \label{eq:full-shg-u2-pml}
\end{align}
Because of the difficulty of implementing the DtN map directly in
NGSolve, we instead truncate the computational domain using the standard
cylindrical-coordinate-stretching PML implemented in
NGSolve~\cite{Collino}. This also motivates solving directly for the
scattered field \(u_1^s\), since the incident field should not be
continued into the PML. Let \(a_{j,\mathrm{PML}}\) denote the
PML-modified finite element form corresponding to the \(j\)th field.
The discrete problem is then to find
\(
    (u_{1,h}^s,u_{2,h})\in V_h\times V_h
\)
such that
\begin{align}
    a_{1,\mathrm{PML}}(u_{1,h}^s,\xi_h)
    &=
    \int_{D}
        \chi_1\overline{(u_{1,h}^s+u^i)}
        u_{2,h}\overline{\xi_h}\,dx
    +
    \int_D
        \kappa_1^2(n_1-1)u^i\overline{\xi_h}\,dx
    &&\forall \xi_h\in V_h,
    \label{eq:discrete-shg-u1-pml}\\
    a_{2,\mathrm{PML}}(u_{2,h},\eta_h)
    &=
    \int_{D}
        \chi_2(u_{1,h}^s+u^i)^2\overline{\eta_h}\,dx
    &&\forall \eta_h\in V_h,
    \label{eq:discrete-shg-u2-pml}
\end{align}
where 
\begin{equation}\label{eqn:uinum}u^i(\mathbf{x})=\exp(i\kappa_1 x_1)\end{equation}

Above, the right hand side  integrals in \eqref{eq:discrete-shg-u1-pml}-
\eqref{eq:discrete-shg-u2-pml} are restricted to \(D\) because
\(n_1=1\) and $\chi_j=0$, $j=1,2$ outside $D$.

For this example, we take \(D=B_1\), with
\[
\kappa_1=8,\qquad\kappa_2=16,\qquad 
    \chi_1=20,
    \qquad
    \chi_2=16,
    \qquad
    n_1=1.5,
    \mbox{ and }
    n_2=1.4
    \quad\text{in }D.
\]
The interface \(\Gamma_R\) between the physical domain and the PML is
placed two fundamental wavelengths from \(\partial D\), where
\(
    \lambda_1=\frac{2\pi}{\kappa_1}.
\)
\begin{figure}[t]
    \begin{center}
    \begin{tabular}{cc}
        \resizebox{0.4\textwidth}{!}{\includegraphics{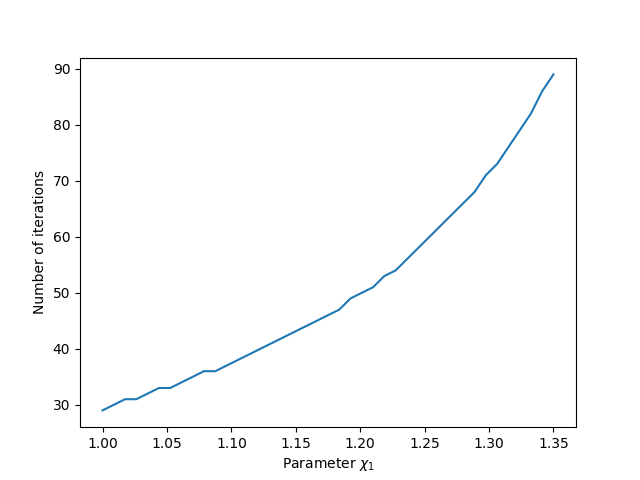}}
            \end{tabular}
    \end{center}
\caption{Number of fixed-point iterates as a function of $\chi_1$ for fixed $\chi_2=10$ in the 3D example.}
\label{fig5}
\end{figure}
The PML has thickness \(\lambda_1/2\), and the NGSolve absorption
parameter is set to \(2i\). A Dirichlet boundary condition is used on the outer PML boundary. We request a triangular mesh with maximum
element diameter \(\lambda_2/8\), where
\(\lambda_2=2\pi/\kappa_2\), and use cubic finite elements
(\(p=3\)). Curved boundaries are represented by a degree-three geometric
mapping. The mesh resolution is determined by the shorter wavelength of
the secondary field \(u_2\), whereas the location and thickness of the
PML must be chosen relative to the longer wavelength of the fundamental
field.

To verify that the PML parameters provide sufficient absorption, we
first consider the corresponding linear problem. We take the incident
plane wave
\(
    u^i=\exp(i\kappa_1x)
\)
and set \(\chi_1=\chi_2=0\). The scattered field \(u_1^s\) then solves
the standard linear scattering problem for a penetrable disc, for which
a series solution is available. Comparing the finite element far-field
pattern $u_{1,h,\infty}^s$ with this series solution $u_{1,\infty}^s$ gives the relative error
\[
    \frac{
        \|u_{1,h,\infty}^s-u_{1,\infty}^s\|_{L^2(0,2\pi)}
    }{
        \|u_{1,\infty}^s\|_{L^2(0,2\pi)}
    }
    =
    8.2\times10^{-6}.
\]
Here the finite element far field pattern is computed using the technique in \cite{sul98} without adaptivity.  Boundary integrals are computed on the curved boundary using NGSolves {\tt Integrate} function with order set to 4.
This agreement indicates that the PML sufficiently absorbs outgoing
waves without noticeably contaminating the solution in the physical
domain. For the full nonlinear problem, the fixed-point iteration reaches the
same stopping criterion used in the preceding experiment after 13
iterations. The resulting fields are shown in Fig.~\ref{fig3}. The
corresponding nonlinear far-field patterns are shown in
Fig.~\ref{fig4}; for comparison, the left panel also includes the
fundamental far-field pattern for the linear problem
\(\chi_1=\chi_2=0\).

\subsection{A three-dimensional example}
We conclude with a three-dimensional example computed using the same
fixed-point iteration as in the preceding experiments. We take the
fundamental wavenumber to be \(\kappa_1=2\), set \(D=B_1\), and choose $n_1=1.5$, $n_2=1.4$.
The incident field is given by (\ref{eqn:uinum}) now viewed as a plane wave in 3D.
The parameters $\chi_j$, $j=1,2$ are constant in $D$ and zero outside. 
 
In order to compute the direct solution of the linear system at each step within the given memory, we used quadratic finite elements
(\(p=2\)) and requested a mesh with maximum element diameter
\(\lambda_2/8\). The PML is also modified compared to the 2D case, beginning one-half of a fundamental wavelength
from \(\partial B_1\) with a thickness equal to one-quarter of that
wavelength. As before, the PML absorption parameter is set to \(2i\). A
homogeneous Dirichlet condition is imposed on the outer boundary of the
PML. Because this PML is both closer to the scatterer and narrower than
the one used in the preceding two-dimensional experiment, reflections
from its outer boundary are expected to have a greater influence on the
computed solution. 

This example is intended to illustrate qualitatively how the number of
fixed-point iterations increases as the strength of the nonlinearity
increases.   We fix \(\chi_2=10\) and record the number of iterations
required to satisfy the convergence criterion for
\(
    \chi_1\in[1,1.35].
\)
The results are shown in Fig.~\ref{fig5}. By
\eqref{eq:small-data-constants}, for a fixed radius $\rho$, the sufficient
contraction factor $q_\rho$ is proportional to
\(\|\chi_1\|_{L^\infty(B_R)}\|\chi_2\|_{L^\infty(B_R)}\). Thus the theory
predicts a deterioration of the contraction bound as \(\chi_1\) increases
with \(\chi_2\) fixed. Since this computation uses a PML rather than the
exact DtN map analyzed above, the comparison is qualitative. Consistently,
the observed number of iterations grows rapidly as \(\chi_1\) increases.
\section{Conclusion}\label{sec:conclusion}
We have analyzed a conforming finite element approximation of a second-harmonic
generation scattering problem modeled by two coupled nonlinear Helmholtz equations.
Using fixed-point techniques, we proved existence and uniqueness of the continuous and discrete
nonlinear solutions under the product-type smallness condition
\eqref{eq:coefficient-smallness}. In the same regime, we derived
a priori finite element error estimates in
the \(H^1\)-norm, and obtained
convergence estimates for the fixed-point iterates used to compute the discrete
solution. The numerical experiments support the predicted convergence rates and illustrate
the behavior of the fixed-point method. They also highlight an important practical
issue for SHG scattering simulations: the size of the computational domain is determined by
the wavelength of the fundamental mode, while the mesh resolution must be fine
enough to resolve the higher-frequency second harmonic. This separation of scales
can significantly increase memory and computational costs. A natural direction for
future work is therefore to use different meshes, polynomial degrees, or adaptive
refinement strategies for the two fields.

\subsubsection*{Declaration of generative AI and AI-assisted technologies in the manuscript preparation process.}
In preparing this work, the authors used generative artificial intelligence tools, specifically OpenAI’s ChatGPT, to identify potential proof strategies, scrutinize arguments for possible errors or gaps, and refine the exposition. All AI-generated outputs were independently evaluated and treated solely as unverified suggestions. The authors retained full control over all decisions concerning the manuscript and accept complete responsibility for its content.


\end{document}